\documentclass[11pt,final]{amsart}
\usepackage[utf8]{inputenc}
\usepackage[T1]{fontenc}
\usepackage[letterpaper,centering]{geometry}
\usepackage{lmodern}
\usepackage{eucal}
\usepackage{mathrsfs}
\usepackage{enumitem}
\usepackage{amsmath,amssymb,amsfonts}
\usepackage{xcolor}
\definecolor{darkblue}{rgb}{0,0,0.3}
\definecolor{darkgreen}{rgb}{0,0.4,0}
\usepackage[colorlinks=true,linkcolor=darkblue, citecolor=darkgreen, unicode,pdfborder={0 0 0}]{hyperref}
\usepackage[ps,arrow,matrix,tips,line]{xy}
{\setbox0\hbox{$ $}}\fontdimen16\textfont2=\fontdimen17\textfont2
\entrymodifiers={+!!<0pt,\the\fontdimen22\textfont2>}
\SelectTips{cm}{11}
\setlist[enumerate]{label={\upshape(\arabic*)},topsep=.7ex}

\setlist[itemize]{leftmargin=*}
\setlist[enumerate]{leftmargin=*}

\usepackage{ifdraft}
\ifdraft{
\usepackage[notcite,notref]{showkeys}
\addtolength{\hoffset}{1.5cm}

}{}

\theoremstyle{plain}
\newtheorem{thm}{Theorem}[section]
\newtheorem{conj}[thm]{Conjecture}

\newtheorem{lem}[thm]{Lemma}
\newtheorem{cor}[thm]{Corollary}
\newtheorem{prop}[thm]{Proposition}
\theoremstyle{definition}

\newtheorem{defn}[thm]{Definition}
\newtheorem{rmk}[thm]{Remark}
\newtheorem{rmks}[thm]{Remarks}
\newtheorem{example}[thm]{Example}

\numberwithin{equation}{section}

\input cyracc.def
\DeclareFontFamily{U}{russian}{}
\DeclareFontShape{U}{russian}{m}{n}
        { <5><6> wncyr5
        <7><8><9> wncyr7
        <10><10.95><12><14.4><17.28><20.74><24.88> wncyr10 }{}
\DeclareSymbolFont{Russian}{U}{russian}{m}{n}
\DeclareSymbolFontAlphabet{\mathcyr}{Russian}
\makeatletter
\let\@math@cyr\mathcyr
\renewcommand{\mathcyr}[1]{\@math@cyr{\cyracc #1}}
\makeatother

\newcommand{\Ba}{{\mathcyr{B}}}

\newcommand{\isoto}{\myxrightarrow{\,\sim\,}}
\makeatletter
\def\myrightarrow{{\setbox\z@\hbox{$\rightarrow$}\dimen0\ht\z@\multiply\dimen0 6\divide\dimen0 10\ht\z@\dimen0\box\z@}}
\def\myrightarrowfill@{\arrowfill@\relbar\relbar\myrightarrow}
\newcommand{\myxrightarrow}[2][]{\ext@arrow 0359\myrightarrowfill@{#1}{#2}}
\makeatother
\newcommand{\mtilde}{{\mathchoice
    {\widetilde{m}}
    {\widetilde{m}}
    {\rlap{$\scriptscriptstyle{m}$}\vphantom{\raise0pt\hbox{$m$}}\smash{\lower2.5pt\hbox{$\scriptscriptstyle\widetilde{\phantom{\scriptscriptstyle{m}}}$}}}
    {\rlap{$\scriptscriptstyle{m}$}\vphantom{\raise.2pt\hbox{$m$}}\smash{\lower2.05pt\hbox{$\scriptscriptstyle\widetilde{\phantom{\scriptscriptstyle{m}}}$}}}}}

\newcommand{\Mtilde}{{\mathchoice
    {\rlap{$M$}\mkern1mu\smash[b]{\lower.5pt\hbox{$\widetilde{\phantom{M}}$}}\mkern-1mu}
    {\rlap{$M$}\mkern1mu\smash[b]{\lower.5pt\hbox{$\widetilde{\phantom{M}}$}}\mkern-1mu}
    {\rlap{$\scriptstyle{M}$}\mkern1mu\smash[b]{\lower.5pt\hbox{$\widetilde{\phantom{\scriptstyle{M}}}$}}\mkern-1mu}
    {\widetilde{M}}}}

\newcommand{\pr}{{\mathrm{pr}}}

\newcommand{\et}{{\mathrm{\acute et}}}

\newcommand{\sO}{{\mathscr O}}

\newcommand{\sU}{{\mathscr U}}

\newcommand{\A}{{\mathbf A}}

\newcommand{\F}{{\mathbf F}}

\renewcommand{\P}{{\mathbf P}}

\newcommand{\Z}{{\mathbf Z}}

\newcommand{\nr}{\mathrm{nr}}

\newcommand{\Gm}{\mathbf{G}_\mathrm{m}}
\newcommand{\Gmbark}{\mathbf{G}_{\mathrm{m},\bar k}}
\newcommand{\Aut}{\mathrm{Aut}}
\newcommand{\Out}{\mathrm{Out}}
\newcommand{\Gal}{\mathrm{Gal}}
\newcommand{\SL}{\mathrm{SL}}
\newcommand{\Pic}{\mathrm{Pic}}
\newcommand{\Br}{\mathrm{Br}}
\newcommand{\Spec}{\mathrm{Spec}}

\renewcommand{\phi}{\varphi}
\renewcommand{\emptyset}{\varnothing}

\newcommand{\Ker}{{\mathrm{Ker}}}

\newcommand{\Hom}{{\mathrm{Hom}}}

\newcommand{\mmu}{\boldsymbol{\mu}}

\newcommand{\chapeau}{{\rlap{\smash{\hbox{\lower4pt\hbox{\hskip1pt$\widehat{\phantom{u}}$}}}}}}
\newcommand{\Picplushat}{\Pic_+^{{\smash{\hbox{\lower4pt\hbox{\hskip0.4pt$\widehat{\phantom{u}}$}}}}}}
\newcommand{\PicplusAhat}{\Pic_{+,\A}^{{\smash{\hbox{\lower4pt\hbox{\hskip.4pt$\widehat{\phantom{u}}$}}}}}}

\newcommand{\Pichat}{\Pic^{{\smash{\hbox{\lower4pt\hbox{\hskip0.4pt$\widehat{\phantom{u}}$}}}}}}

\newcommand{\Tr}{\mathrm{Tr}}
\newcommand{\Id}{\mathrm{Id}}

 \makeatletter
 \renewcommand{\tocsection}[3]{%
   \indentlabel{\@ifnotempty{#2}{\bfseries\ignorespaces#1 #2\quad}}\bfseries#3}

 \renewcommand{\tocsubsection}[3]{%
   \indentlabel{\@ifnotempty{#2}{\hspace{1.6em}\ignorespaces#1 #2\quad}}#3}
 \makeatother

\makeatletter\let\@wraptoccontribs\wraptoccontribs\makeatother

\date{June 9th, 2022; revised on May 2nd, 2023}
\title{Supersolvable descent for rational points}

\author{Yonatan Harpaz}
\address{Institut Galil\'ee, Universit\'e Sorbonne Paris Nord, 99~avenue Jean-Baptiste Cl\'ement, 93430 Villetaneuse, France}
\email{harpaz@math.univ-paris13.fr}

\author{Olivier Wittenberg}
\address{Institut Galil\'ee, Universit\'e Sorbonne Paris Nord, 99~avenue Jean-Baptiste Cl\'ement, 93430 Villetaneuse, France}
\email{wittenberg@math.univ-paris13.fr}

\begin{document}

\begin{abstract}
We construct an analogue of the classical descent theory of Colliot-Thélène
and Sansuc in which algebraic tori are replaced with finite supersolvable
groups. As an application, we show that rational points are dense in the
Brauer--Manin set for smooth compactifications of certain quotients of
homogeneous spaces by finite supersolvable groups.  For suitably chosen
homogeneous spaces, this implies the existence of supersolvable Galois
extensions of number fields with prescribed norms, generalising work of
Frei--Loughran--Newton.
\end{abstract}

\vspace*{2cm}
\maketitle

\section{Introduction}

Let~$X$ be a smooth and irreducible variety over a number field~$k$.
The study of rational points on~$X$ often begins by embedding the set~$X(k)$
 diagonally into the product
$X(k_{\Omega})= \prod_{v \in \Omega} X(k_v)$,
where~$\Omega$ denotes the set of places of~$k$ and~$k_v$ the completion of~$k$ at~$v$.
We endow~$X(k_\Omega)$
 with the product of the $v$\nobreakdash-adic topologies.
The weak approximation property, that is, the density of~$X(k)$ in~$X(k_\Omega)$, frequently fails.
Following Manin~\cite{maninicm}, one can attempt to explain such failures
 by considering the Brauer--Manin set
 $X(k_{\Omega})^{\Br_{\nr}(X)}$,
defined as the set of elements of~$X(k_{\Omega})$ that are orthogonal,
with respect to the Brauer--Manin pairing, to the unramified Brauer group~$\Br_{\nr}(X)$ of~$X$.
We recall that $\Br_{\nr}(X)$ is the subgroup of~$\Br(X)$ formed by
those classes that extend to any (equivalently, to some) smooth compactification of~$X$,
and that
the Brauer--Manin set $X(k_{\Omega})^{\Br_{\nr}(X)}$ is a closed subset of~$X(k_\Omega)$
that satisfies the inclusions $X(k) \subseteq X(k_{\Omega})^{\Br_{\nr}(X)} \subseteq X(k_{\Omega})$
(see \cite[\textsection5.2]{skobook}).
A conjecture of Colliot-Thélène predicts that the Brauer--Manin
set is enough to fully account for
the gap between the topological closure of $X(k)$ and $X(k_{\Omega})$
when~$X$ is
\emph{rationally connected}---that is, when for any algebraically closed field extension~$K$ of~$k$,
two general $K$\nobreakdash-points of~$X$ can be joined by a rational curve over~$K$:

\newcommand{\citeconjct}{\cite{ctbudapest}}
\begin{conj}[\citeconjct]
\label{conj:ct}
Let~$X$ be a smooth and rationally connected variety over a number field $k$. The set $X(k)$ is a
dense subset of $X(k_{\Omega})^{\Br_{\nr}(X)}$.
\end{conj}

Though this conjecture is wide open in general, it has been established in several special cases.
One approach is via the theory of descent developed by Colliot-Thélène and Sansuc.
To explain it,
let us assume for the moment that~$X$ is proper. 
Given an algebraic torus~$T$ over $k$, this theory considers torsors $Y \to X$ under $T$.
The \emph{type} of such a torsor is the isomorphism class of the torsor
obtained from it by extending the scalars from~$k$ to an algebraic closure~$\bar k$ of~$k$;
a \emph{type}, in this context, is defined to be an isomorphism class of torsors over~$X_{\bar k}$, under~$T_{\bar k}$,
that is invariant under $\Gal(\bar k/k)$.
Torsors $Y\to X$ under~$T$ are classified, up to isomorphism, by the étale cohomology
group $H^1_\et(X,T)$, and types are the elements of the abelian group $H^1_\et(X_{\bar k},T_{\bar k})^{\Gal(\bar k/k)}$.
As~$X$ is proper, 
these groups fit into the exact sequence
\begin{align}
\label{se:hstype}
H^1(k,T) \to H^1_\et(X,T) \to H^1_\et(X_{\bar k},T_{\bar k})^{\Gal(\bar k/k)} \to H^2(k,T) \to H^2_\et(X,T)
\end{align}
induced by the Hochschild--Serre spectral sequence.
As can be seen from this sequence, if a type comes from a torsor $Y \to X$ defined over~$k$, then the isomorphism
class of this torsor
is unique up to a twist by an element of $H^1(k,T)$. In general, not every type comes from a torsor over~$k$:
the map $H^1_\et(X,T) \to H^1_\et(X_{\bar k},T_{\bar k})^{\Gal(\bar k/k)}$
need not be surjective. It is surjective if $X(k)\neq\emptyset$, as the map
$H^2(k,T) \to H^2_\et(X,T)$ then possesses a retraction.

The following foundational theorem
fully describes the \emph{algebraic} Brauer--Manin set
$X(k_\Omega)^{\Br_{1,\nr}(X)}$ in terms of the arithmetic of torsors under a torus, of a given type, over~$X$.
By definition
$X(k_\Omega)^{\Br_{1,\nr}(X)}$
is the subset of $X(k_\Omega)$ consisting of those collections of local points that are orthogonal,
with respect to the Brauer--Manin pairing, to the
\emph{algebraic} unramified Brauer
group $\Br_{1,\nr}(X)=\Ker\mkern1mu\big(\Br_{\nr}(X) \to \Br(X_{\bar k})\big)$.

\begin{thm}[Colliot-Thélène--Sansuc \cite{ctsandescent2}]
\label{thm:ct-sansuc}
Let~$X$ be a smooth, proper and geometrically irreducible variety over a number field~$k$
such that $\Pic(X_{\bar k})$ is torsion-free.
Let~$T$ be an algebraic torus over~$k$.
Let $\lambda \in H^1_\et(X_{\bar k},T_{\bar k})^{\Gal(\bar k/k)}$.
Then
\begin{align}
X(k_{\Omega})^{\Br_{1,\nr}(X)}= \bigcup_{f: Y \to X} f\Big(Y(k_{\Omega})^{\Br_{1,\nr}(Y)}\Big)\rlap,
\end{align}
where $f:Y\to X$ ranges over
the isomorphism classes of torsors $Y \to X$ of type $\lambda$.
In particular, if $Y(k)$ is a dense subset of $Y(k_{\Omega})^{\Br_{1,\nr}(Y)}$ for every torsor $Y \to X$ of type~$\lambda$, then $X(k)$ is a dense subset of $X(k_{\Omega})^{\Br_{1,\nr}(X)}$.
\end{thm} 

As~$X$ is assumed to be proper in Theorem~\ref{thm:ct-sansuc},
one has $\Br_{1,\nr}(X)=\Br_1(X)$.
The group  $\Br_{1,\nr}(Y)$, on the other hand,
may be smaller
than $\Br_1(Y)$, since~$Y$ is not proper.

It is by now understood that
 Theorem~\ref{thm:ct-sansuc} still holds if $\Pic(X_{\bar k})$ is allowed to contain torsion
(see e.g.~\cite{weiopen}).
When $\Pic(X_{\bar k})$ is torsion-free, however,
there exists a privileged type of torsors over~$X$:
denoting by~$T'$ the algebraic torus over~$k$ with character group $\Pic(X_{\bar k})$,
there is a canonical isomorphism $H^1_\et(X_{\bar k},T'_{\bar k})=\mathrm{End}(\Pic(X_{\bar k}))$;
the torsors $Y' \to X$ under~$T'$
whose type is classified by the identity endomorphism of~$\Pic(X_{\bar k})$ are called \emph{universal torsors}.
They enjoy the special property that $\Pic(Y'_{\bar k})=0$.
By the Hochschild--Serre spectral sequence, it follows that the
natural map $\Br(k)\to \Br_1(Y')$ is surjective, so that $Y'(k_{\Omega})^{\Br_{1,\nr}(Y')}=Y'(k_\Omega)$.
As a consequence, Theorem~\ref{thm:ct-sansuc} effectively
reduces the statement that $X(k)$ is dense in $X(k_{\Omega})^{\Br_{1,\nr}(X)}$ to the
(in principle simpler) weak approximation property for the universal torsors of~$X$.
This approach
was fruitfully carried out in many special cases, notably for Châtelet surfaces in the influential
two-part work~\cite{cssI,cssII}, and later for various other types of varieties
\cite{heathbrownskorobogatov,
cthasko,browningmatthiesen,bms, derenthalsmeetswei,skodescenttoric}.

We note that
even though
Theorem~\ref{thm:ct-sansuc} is stated and proved in~\cite{ctsandescent2} only
in the case of universal torsors,
the general case follows.  Indeed, for any~$T$ and~$\lambda$ as in Theorem~\ref{thm:ct-sansuc},
the type~$\lambda$ determines a morphism $T' \to T$, so that any universal torsor $Y' \to X$
factors through a torsor $Y\to X$ of type~$\lambda$, and the image
of $Y'(k_\Omega)=Y'(k_{\Omega})^{\Br_{1,\nr}(Y')}$
 in $Y(k_\Omega)$
is then contained in
$Y(k_{\Omega})^{\Br_{1,\nr}(Y)}$ by the projection formula.

The classical theory of descent of Colliot-Thélène and Sansuc allows one to
neatly capture the algebraic part of the Brauer--Manin obstruction
in terms of universal torsors and their local points. Formulated as above, it admits, however, a
limitation:
when the strategy consisting in applying Theorem~\ref{thm:ct-sansuc} to verify Conjecture~\ref{conj:ct}
for a given~$X$ works,
one finds that a stronger claim than Conjecture~\ref{conj:ct} is in fact being proved, namely, the density of $X(k)$ 
not only in
$X(k_{\Omega})^{\Br_{\nr}(X)}$
but also  in the larger set  $X(k_{\Omega})^{\Br_{1,\nr}(X)}$. The two sets
$X(k_{\Omega})^{\Br_{\nr}(X)}$
and  $X(k_{\Omega})^{\Br_{1,\nr}(X)}$
 coincide for geometrically rational varieties, but among rationally connected varieties
it does happen that~$X(k)$ fails to be dense in the larger one
(see \cite[\textsection2]{hartranscendant}, \cite[Example~5.4]{dlan}),
thus limiting the scope of applicability of the method.
A way to overcome this issue was suggested
in~\cite{hwzceh}, where the following variant of
Theorem~\ref{thm:ct-sansuc} is proved---to be precise, Theorem~\ref{thm:zceh-descent} results from combining
\cite[Théorème~2.1]{hwzceh} in the proper case with \cite[Theorem~6.1.2~(a)]{skobook}:

\begin{thm}[\cite{hwzceh}]\label{thm:zceh-descent}
Let~$X$ be a smooth, proper and geometrically irreducible variety over a number field~$k$.
Let~$T$ be an algebraic torus over~$k$.
Let $\lambda \in H^1_\et(X_{\bar k},T_{\bar k})^{\Gal(\bar k/k)}$.
Then
\begin{align}
\label{eq:zceh-descent}
X(k_{\Omega})^{\Br_{\nr}(X)}= \bigcup_{f: Y \to X} f\Big(Y(k_{\Omega})^{\Br_{\nr}(Y)}\Big)\rlap,
\end{align}
where $f:Y\to X$ ranges over
the isomorphism classes of torsors $Y \to X$ of type $\lambda$.
In particular, if $Y(k)$ is a dense subset of $Y(k_{\Omega})^{\Br_{\nr}(Y)}$ for every torsor $Y \to X$ of type~$\lambda$, then $X(k)$ is a dense subset of $X(k_{\Omega})^{\Br_{\nr}(X)}$.
\end{thm} 

The theorems discussed so far admit generalisations to open varieties in which $X(k_\Omega)$
is replaced with the space of adelic points $X(\A_k)$, and $\Br_{\nr}(X)$ with the larger group~$\Br(X)$;
for instance, Theorem~\ref{thm:ct-sansuc} was generalised in
the work of Wei~\cite{weiopen},
which builds on
\cite[Theorem~6.1.2]{skobook} and on \cite[Proposition~8.12]{harskoopen}, and later in the work
of Cao, Demarche and Xu~\cite[Theorem~1.2]{cdx}.
We shall not consider such generalisations in this article.
Theorems~\ref{thm:ct-sansuc} and~\ref{thm:zceh-descent}
can also be extended to
open varieties while keeping focus on the set~$X(k_\Omega)$ and on the group $\Br_{\nr}(X)$.
In particular, it follows from \cite[Proposition~8.12, Corollary~8.17]{harskoopen},
from Harari's ``formal lemma'' \cite[Theorem~13.4.3]{ctskobook}
and from \cite[Théorème~2.1]{hwzceh},
as in the proof of \cite[Corollaire~2.2]{hwzceh},
that
the statements of 
Theorem~\ref{thm:ct-sansuc} and Theorem~\ref{thm:zceh-descent}
remain valid without the properness assumption,
provided one assumes, in the case of Theorem~\ref{thm:zceh-descent}, that the quotient of
$\Br_\nr(X)$ by the subgroup of constant classes is finite
(as is the case when~$X$ is rationally connected),
and provided one replaces, on the one hand, the notion of type due to Colliot-Thélène and Sansuc
by that of extended type introduced by Harari and Skorobogatov~\cite{harskoopen},
and on the other hand,
the right-hand sides
of the asserted equalities by their topological closure in~$X(k_\Omega)$.
  Thus, for instance,
the equality~\eqref{eq:zceh-descent}
becomes
\begin{align}
X(k_{\Omega})^{\Br_{\nr}(X)}= \overline{\bigcup_{f: Y \to X} f\Big(Y(k_{\Omega})^{\Br_{\nr}(Y)}\Big)}\rlap,
\end{align}
where $\mkern1mu\overline{\mkern-1muM}$ denotes the topological closure of a subset~$M$ of~$X(k_\Omega)$.

Pursuing the line of thought explored in~\cite{hwzceh}, the
goal of the present article is to prove an analogue of
Theorem~\ref{thm:zceh-descent} in the case where the torus~$T$ is replaced
with a \emph{supersolvable finite group}.
We shall adapt the
notion of a type to the non-abelian setting as follows: for a smooth and geometrically irreducible variety~$X$,
we define a finite descent type on~$X$ to be a variety~$\bar Y$ over~$\bar k$ equipped with
a finite étale map
 $\bar Y \to X_{\bar k}$ such that the composed map $\bar Y \to X_{\bar k} \to X$ is
Galois, in the sense that the function
field extension $\bar k(\bar Y)/k(X)$ is Galois; and we define a torsor of type~$\bar Y$ to be
an étale map $Y \to X$ such
that the $X_{\bar k}$\nobreakdash-schemes $Y_{\bar k}$ and~$\bar Y$ are isomorphic.
A torsor of type~$\bar Y$ is then naturally a torsor under a finite group scheme over~$k$
whose group of $\bar k$\nobreakdash-points is the finite group $\Aut(\bar Y/X_{\bar k})$
canonically associated with the type~$\bar Y$
(although the group scheme itself is not canonically associated with~$\bar Y$).
Any finite descent type determines not only a finite group $\bar G = \Aut(\bar Y/X_{\bar k})$
but also an outer Galois action $\Gal(\bar k/k) \to \Out(\bar G)$ on $\bar G$. We say that a finite descent
type~$\bar Y$ is supersolvable if~$\bar G$ admits a filtration
\begin{align*}
\{1\} =  \bar G_0 \subseteq  \bar G_1 \subseteq \dots \subseteq  \bar G_n =  \bar G
\end{align*}
such that each $\bar G_i$ is a normal subgroup of~$\bar G$ stable under the outer Galois action, and each successive quotient $\bar G_{i+1}/\bar G_i$ is cyclic.
We say that it is rationally connected if~$\bar Y$, as a variety over~$\bar k$,
is rationally connected (in the sense 
made explicit just before Conjecture~\ref{conj:ct}).

Our main result in this article is the following:

\begin{thm}[supersolvable descent, see Theorem~\ref{th:main}]\label{th:main-intro}
Let $X$ be a smooth and geometrically irreducible variety over a number field~$k$.
Let~$\bar Y$ be a rationally connected supersolvable finite descent type on~$X$.
Then
\begin{align}
X(k_{\Omega})^{\Br_{\nr}(X)}= \overline{\bigcup_{f: Y \to X} f\Big(Y(k_{\Omega})^{\Br_{\nr}(Y)}\Big)}\rlap,
\end{align}
where $f:Y\to X$ ranges over
the isomorphism classes of torsors $Y \to X$ of type~$\bar Y$.
In particular, 
 Conjecture~\ref{conj:ct} holds for~$X$
if it holds for~$Y$
for every torsor $Y \to X$ of type~$\bar Y$.
\end{thm} 

It is now crucial that~$X$ is not assumed proper in the statement of Theorem~\ref{th:main-intro}:
indeed, the hypothesis that~$\bar Y$ is rationally connected
implies that~$X_{\bar k}$ is rationally connected, and
proper smooth rationally
connected varieties are simply connected and hence do not possess
nontrivial finite descent types.

Theorem~\ref{th:main-intro} can be used to prove Conjecture~\ref{conj:ct} for quotient varieties $Y/G$ where $G$ is a finite supersolvable group acting freely on a quasi-projective rationally connected variety~$Y$, when
Conjecture~\ref{conj:ct} is already known for~$Y$ and for all of its twists. In particular, the following case of Conjecture~\ref{conj:ct} can be proved via supersolvable descent (see Example~\ref{example:quotientlinear}):

\begin{cor}
Let~$Y$ be a homogeneous space of a connected linear algebraic group~$L$
with connected geometric stabilisers.
Suppose that a finite
supersolvable group~$G$  acts on~$L$ (as an algebraic group) and
on~$Y$ (as a homogeneous space of~$L$, compatibly with its action on~$L$), and
that the action on~$Y$ is free.  Then
Conjecture~\ref{conj:ct} holds for the variety $X = Y/G$.
\end{cor}

As an application, we prove the following generalisation of a theorem of 
Frei, Loughran and Newton~\cite{freiloughrannewton}:

\begin{cor}
Let~$k$ be a number field and $\mathcal A \subset k^*$ be a finitely generated subgroup.
Let~$G$ be a supersolvable finite group.
Then there exists a Galois extension $K/k$ with Galois group isomorphic to~$G$
such that every element of~$\mathcal A$ is a norm from~$K$.
Moreover, given a finite set of places~$S$ of~$k$, one can require that the places of~$S$ split in~$K$.
\end{cor}

More applications are described in~\S\ref{sec:applications}.

\subsection{Notation and terminology}\label{subsec:notation}

We fix once and for all a field~$k$ of characteristic~$0$
and an algebraic closure~$\bar k$ of~$k$.
A \emph{variety} is a separated scheme of finite type over a field.
Somewhat unconventionally,
we shall say that a variety~$X$
over~$k$ is \emph{rationally connected}
if the smooth proper varieties over~$\bar k$ that are birationally equivalent to~$X_{\bar k}$ are rationally connected
in the sense of Campana, Kollár, Miyaoka and Mori
(see \cite[Chapter~IV]{kollarbook}).
If~$X$ is a smooth irreducible variety over~$k$,
we denote by $\Br_{\nr}(X) \subseteq \Br(X)$ the unramified Brauer group of~$k(X)/k$
(see \cite[\textsection6.2]{ctskobook}).

The words ``torsor'', ``action'', ``homogeneous space'' will refer
to left torsors, left actions, left homogeneous spaces, unless indicated otherwise.
An \emph{outer action} of a profinite group~$\Gamma$ on a discrete group~$H$
is a continuous group homomorphism  $\Gamma \to \Out(H)$,
where we endow the group $\Out(H)$ of outer automorphisms of~$H$ with
the topology induced by the compact-open topology on~$\Aut(H)$.

When~$G$ is an algebraic group over~$k$, we denote by $H^1(k,G)$
the first non-abelian Galois cohomology pointed set (see \cite{serrecg}).
When we write $[\sigma]$
for
 an element of this set, we mean that~$\sigma$
is a cocycle representing the cohomology class~$[\sigma]$.

When~$k$ is a number field, we denote by~$\Omega$ the set of places of~$k$
and by~$k_v$
the completion of~$k$ at $v \in \Omega$.
For any variety~$X$ over~$k$,
we let $X(k_{\Omega})=\prod_{v \in \Omega}X(k_v)$, endow
this set with the product of the $v$\nobreakdash-adic
topologies, and when~$X$ is smooth and irreducible, we denote by $X(k_{\Omega})^{\Br_{\nr}(X)} \subseteq X(k_\Omega)$ the
Brauer--Manin set (see \cite[\textsection5.2]{skobook}).

\subsection{Acknowledgment}

We are grateful to David Harari, to Dasheng Wei and to the referee for their comments on a first version of this article.

\section{Finite descent types}
\label{sec:finitedescenttypes}

We recall that~$k$ denotes a field of characteristic~$0$.
We fix, until the end of~\textsection\ref{sec:finitedescenttypes},
a smooth and geometrically irreducible variety~$X$ over~$k$.

\subsection{Finite descent types and torsors}
\label{subsec:finitedescenttypes}

If~$G$ is a finite \'etale group scheme over~$k$ and $Y \to X$ is a torsor under~$G$
such that~$Y$ is geometrically irreducible over~$k$, the function field extension $\bar k(Y)/k(X)$
is Galois (with Galois group $G(\bar k)\rtimes \Gal(\bar k/k)$).  This remark motivates the following definition.

\begin{defn}
A \emph{finite descent type on~$X$}
is an irreducible finite étale $X_{\bar k}$\nobreakdash-scheme~$\bar Y$ that is
Galois over~$X$ (i.e.\ such that the function field extension $\bar k(\bar Y)/k(X)$ is Galois).
\end{defn}

Equivalently, a finite descent type on~$X$
is an irreducible finite \'etale Galois $X_{\bar k}$\nobreakdash-scheme~$\bar Y$
such that the natural morphism $\Aut(\bar Y/X) \to \Gal(\bar k/k)$ is surjective.

\begin{defn}
\label{defn:torsoroftype}
Given a finite descent type~$\bar Y$ on~$X$,
a \emph{torsor of type~$\bar Y$} is an $X$\nobreakdash-scheme~$Y$
such that the $X_{\bar k}$\nobreakdash-schemes $Y_{\bar k}$ and~$\bar Y$
are isomorphic.
\end{defn}

Although Definition~\ref{defn:torsoroftype}
does not make reference to a group scheme,
the name ``torsor'' is justified by the remark that
any torsor $Y \to X$ of type~$\bar Y$
in the above sense
is canonically
a torsor under the
finite \'etale group scheme~$G$ over~$k$ defined by $G(\bar k)=\Aut(Y_{\bar k}/X_{\bar k})$.
We warn the reader, though, that two torsors of type~$\bar Y$ are not,
in general, torsors under isomorphic group schemes: the underlying finite group is always isomorphic to
$\Aut(\bar Y/X_{\bar k})$, but
the action of $\Gal(\bar k/k)$ on $\Aut(\bar Y/X_{\bar k})$ depends,
 in general, on the choice of $Y\to X$. As we shall see
in Proposition~\ref{prop:nonabcoh}~(ii), it is nevertheless true that two torsors of type~$\bar Y$ are torsors under group schemes
that are inner forms of each other.

Let~$\bar Y$ be a finite descent type on~$X$.
The short exact sequence of profinite groups
\begin{align}
\label{eq:gautgal}
1 \to \Aut(\bar Y/X_{\bar k}) \to \Aut(\bar Y/X) \to \Gal(\bar k/k) \to 1
\end{align}
induces a
 continuous outer action of $\Gal(\bar k/k)$
on the finite group $\bar G=\Aut(\bar Y/X_{\bar k})$.
This outer action, together with the natural action of $\Gal(\bar k/k)$
on the scheme~$X_{\bar k}$, induces, in turn,
a continuous action of $\Gal(\bar k/k)$ on the pointed set $H^1_\et(X_{\bar k}, \bar G)$
of isomorphism classes of torsors under~$\bar G$
over~$X_{\bar k}$ (notation justified by \cite[Chapter~III, Corollary~4.7]{milneet}).
With respect to this action,
the class of the $\bar G$\nobreakdash-torsor $\bar Y \to X_{\bar k}$ is invariant.

The next two propositions 
provide a group-theoretic point of view on Definition~\ref{defn:torsoroftype}
in terms of the short exact sequence~\eqref{eq:gautgal}.
When we speak of a \emph{splitting} of~\eqref{eq:gautgal}, we shall always mean a continuous homomorphism
$\Gal(\bar k/k)\to \Aut(\bar Y/X)$ which is a section of the projection
$\Aut(\bar Y/X) \to \Gal(\bar k/k)$.

\begin{prop}
\label{prop:splittings}
Let~$\bar Y$ be a finite descent type on~$X$.
\begin{enumerate}[label={\upshape(\roman*)}]
\itemsep=1ex
\item
Splittings of~\eqref{eq:gautgal}
are in one-to-one correspondence with
isomorphism classes of torsors $Y\to X$ of type~$\bar Y$
endowed with
an $X_{\bar k}$\nobreakdash-isomorphism $\iota: Y_{\bar k} \isoto \bar Y$.
\item
Splittings of~\eqref{eq:gautgal} up to conjugation by
$\Aut(\bar Y/X_{\bar k})$ are in one-to-one correspondence with
isomorphism classes of torsors $Y\to X$ of type~$\bar Y$.
\item
Splittings of~\eqref{eq:gautgal} exist if~$X(k)\neq\emptyset$.
\end{enumerate}
\end{prop}

\begin{proof}
Assertion~(i) results from the fact that such pairs~$(Y,\iota)$ correspond
to subextensions $k(X) \subset L \subset \bar k(\bar Y)$ such that $L \cap \bar k=k$
and $\bar k(\bar Y)=L\bar k$, and from Galois theory.
Assertion~(ii) follows from~(i).
For~(iii), see \cite[Proposition~2.5]{wittenbergslc}.
\end{proof}

\begin{prop}
\label{prop:nonabcoh}
Let~$\bar Y$ be a finite descent type on~$X$
and set $\bar G=\Aut(\bar Y/X_{\bar k})$.
\begin{enumerate}[label={\upshape(\roman*)}]
\itemsep=1ex
\item
Let us fix a pair $(Y,\iota)$ as in Proposition~\ref{prop:splittings}~(i) corresponding to a splitting~$s$
of~\eqref{eq:gautgal}.
Let~$G$ be the finite \'etale group scheme over~$k$ defined by $G(\bar k)=\bar G$,
with the continuous action of $\Gal(\bar k/k)$ on~$\bar G$ by conjugation through~$s$.
Then the natural action of~$\bar G$ on~$\bar Y$ descends to an action of~$G$ on the $X$\nobreakdash-scheme~$Y$,
making it a torsor under~$G$.
\item
Let us fix another pair $(Y',\iota')$, corresponding to another splitting~$s'$
of~\eqref{eq:gautgal}.
Let~$G'$ denote the corresponding group scheme, as in~(i).
Let~$\sigma$ be the cocycle
\begin{align*}
s's^{-1}:\Gal(\bar k/k)\to G(\bar k)\rlap{.}
\end{align*}
There are compatible canonical isomorphisms of $k$\nobreakdash-group schemes $G'\simeq{}G^\sigma$
 and of $X$\nobreakdash-schemes $Y'\simeq{}Y^\sigma$,
where $G^\sigma$ denotes the inner twist of the group scheme~$G$ by~$\sigma$
and $Y^\sigma$ the twist of the torsor~$Y$ by~$\sigma$
(see \cite[Lemma~2.2.3]{skobook}).
\end{enumerate}
\end{prop}

\begin{proof}
Both assertions follow from unwinding the correspondence
of Proposition~\ref{prop:splittings}~(i) and noting that the natural action of~$\Gal(\bar k/k)$
on the scheme~$Y_{\bar k}$ coincides with the action obtained by transport
of structure, via~$\iota$, from the action of~$\Gal(\bar k/k)$ on~$\bar Y$ given by~$s$.
\end{proof}

\begin{rmk}
Let us assume that~$\bar G$ is abelian.  In this case, the outer action of $\Gal(\bar k/k)$ on~$\bar G$
induced by~\eqref{eq:gautgal} is an actual action,
thanks to which~$\bar G$ canonically descends to a finite \'etale group scheme~$G$ over~$k$.
In addition, any torsor $Y\to X$ of type~$\bar Y$ is canonically a torsor under~$G$
since
any two $X_{\bar k}$\nobreakdash-isomorphisms $Y_{\bar k}\isoto \bar Y$ give rise to the same
isomorphism $\Aut(\bar Y/X_{\bar k}) \simeq \Aut(Y_{\bar k}/X_{\bar k})$.
Thus, the abelian situation is summarised by
the natural map
\begin{align}
\label{eq:hsmap}
H^1_\et(X,G)\to H^0(k,H^1_\et(X_{\bar k},\bar G))\rlap{,}
\end{align}
which appears in the Hochschild--Serre spectral sequence and which sends the isomorphism class of
a $G$\nobreakdash-torsor over~$X$ to (the isomorphism class of) its \emph{type}.  We see, in particular,
that the terminology of Definition~\ref{defn:torsoroftype}
is consistent with the one introduced by
Colliot-Thélène and Sansuc \cite[\textsection2]{ctsandescent2}
in the theory of descent under groups of multiplicative type,
as well as with the notion of extended type of
Harari and Skorobogatov~\cite{harskoopen} in the finite case.
\end{rmk}

\subsection{Supersolvability}

The notion of supersolvability for finite groups endowed with an outer action
of $\Gal(\bar k/k)$,
first introduced in
 \cite[Définition~6.4]{hwzceh}, plays
a central r\^ole in the present article.
We recall it below.

\begin{defn}
\label{def:supersolvable}
A finite group~$\bar G$ endowed with an outer action of $\Gal(\bar k/k)$
is said to be
\emph{supersolvable} if there exist an integer~$n$
and
a sequence
\begin{align*}
\{1\} = \bar G_0 \subseteq \bar G_1 \subseteq \dots \subseteq \bar G_n = \bar G
\end{align*}
of normal subgroups of~$\bar G$
such that
for all $i \in \{1,\dots,n\}$, the quotient
 $\bar G_i/\bar G_{i-1}$ is cyclic
and the subgroup~$\bar G_i$ is stable under the outer action of $\Gal(\bar k/k)$.
(A normal subgroup is said to be stable under an outer automorphism
if it is stable under any automorphism that lifts the given outer automorphism.
This is independent of the choice of the lift.)

A~finite descent type~$\bar Y$ on~$X$ is said to be \emph{supersolvable} if
the finite group $\Aut(\bar Y/X_{\bar k})$, endowed with the outer action
of $\Gal(\bar k/k)$ induced by~\eqref{eq:gautgal}, is supersolvable in this
sense.

When the integer~$n$ is fixed, we say that~$\bar G$, or~$\bar Y$, is
\emph{supersolvable of class~$n$}.
\end{defn}

\section{Supersolvable descent}

We are now in a position to state and prove our main theorem.
We say that a finite descent type~$\bar Y$ on~$X$ is \emph{rationally connected}
if~$\bar Y$ is a rationally connected variety over~$\bar k$ in the sense of \S\ref{subsec:notation}. 

\begin{thm}
\label{th:main}
Let $X$ be a smooth and geometrically irreducible variety over a number field~$k$.
Let~$\bar Y$ be a rationally connected supersolvable finite descent type on~$X$.
Then
\begin{align}
\label{eq:th:main}
X(k_{\Omega})^{\Br_{\nr}(X)}= \overline{\bigcup_{f: Y \to X} f\Big(Y(k_{\Omega})^{\Br_{\nr}(Y)}\Big)}\rlap,
\end{align}
where $f:Y\to X$ ranges over
the isomorphism classes of torsors $Y \to X$ of type~$\bar Y$
and $\mkern1mu\overline{\mkern-1muM}$ denotes the topological closure of a subset~$M$ of~$X(k_\Omega)$.
\end{thm}

\begin{rmks}
\label{rmk:firstremark}
(i)
The existence of a rationally connected finite descent type on~$X$, assumed in Theorem~\ref{th:main},
implies that~$X$ itself is rationally connected.

(ii)
When~$X$ is rationally connected,
the statement of Theorem~\ref{th:main} can be expected to hold
for an arbitrary finite descent type~$\bar Y$ on~$X$.
Indeed, the equality~\eqref{eq:th:main} would result from Conjecture~\ref{conj:ct}
(see \cite[Proposition~2.5]{wittenbergslc}).

(iii)
By Proposition~\ref{prop:splittings}~(ii) and Proposition~\ref{prop:nonabcoh},
the equality~\eqref{eq:th:main}
can be reformulated as the following
two-part statement:
\begin{itemize}
\item
if $X(k_{\Omega})^{\Br_{\nr}(X)}\neq\emptyset$,
then
the natural outer action of~$\Gal(\bar k/k)$ on
$\bar G=\Aut(\bar Y/X_{\bar k})$
can be lifted to an actual continuous action,
in such a way that if~$G$ denotes the resulting finite \'etale group scheme over~$k$,
the $\bar G$\nobreakdash-torsor $\bar Y \to X_{\bar k}$
descends to a $G$\nobreakdash-torsor
$f:Y\to X$;
\item
fixing such~$G$ and~$Y$
and denoting by $f^\sigma:Y^\sigma \to X$ the
twist of the torsor $f:Y\to X$ by a cocycle~$\sigma$,
one has an equality
\begin{align*}
X(k_\Omega)^{\Br_{\nr}(X)}
=
\overline{\bigcup_{[\sigma] \in H^1(k,G)} f^\sigma\Big(Y^\sigma(k_\Omega)^{\Br_{\nr}(Y^\sigma)}\Big)}\rlap.
\end{align*}
\end{itemize}
\end{rmks}

The proof of Theorem~\ref{th:main} will be given
in~\textsection\textsection\ref{subsec:cyclicdescent}--\ref{subsec:general-case},
first in the case where~$G$ is cyclic in~\S\ref{subsec:cyclicdescent} and
then in the general case in~\S\ref{subsec:general-case}.
Before going into the proof,
let us illustrate Theorem~\ref{th:main} with the following special case.
We say that a finite \'etale group scheme~$G$ over~$k$
is \emph{supersolvable}
if $G(\bar k)$
is supersolvable in the sense of Definition~\ref{def:supersolvable}
with respect to the natural outer action
of $\Gal(\bar k/k)$ (which happens in this case to be an actual action).

\begin{cor}
\label{cor:quotientsRC}
Let~$Y$ be a smooth, quasi-projective rationally connected variety over a number field~$k$.
Let~$G$ be a supersolvable finite \'etale group scheme over~$k$, acting freely on~$Y$.
Let $X = Y/G$ denote the quotient.
If Conjecture~\ref{conj:ct} holds for~$Y^{\sigma}$ for all $[\sigma] \in H^1(k,G)$, then it holds for~$X$.
\end{cor}

\begin{proof}
The projection $Y \to X$ is a $G$\nobreakdash-torsor,
hence~$Y_{\bar k}$ is a supersolvable finite descent type on~$X$,
as remarked at the beginning of~\textsection\ref{subsec:finitedescenttypes}.
The conclusion now follows from Theorem~\ref{th:main}, in view of Remark~\ref{rmk:firstremark}~(iii).
\end{proof}

\begin{example}
\label{example:quotientlinear}
Suppose that~$Y$ is a homogeneous space of a connected linear algebraic
group~$L$ with connected geometric stabilisers, and that the finite
supersolvable group~$G$ acts compatibly on~$L$ as an algebraic group and on~$Y$ as a homogeneous space of~$L$,
with the action on~$Y$ being free.
For
each $[\sigma] \in H^1(k,G)$, the twisted variety~$Y^{\sigma}$ is then a
homogeneous space of the twisted algebraic group~$G^{\sigma}$, with connected
geometric stabilisers. As such, the variety~$Y^\sigma$ satisfies Conjecture~\ref{conj:ct},
according to Borovoi~\cite{borovoi}. It now follows from Corollary~\ref{cor:quotientsRC} that
Conjecture~\ref{conj:ct} holds for the quotient variety $X = Y/G$.
It should be noted that~$X$ is not itself, in general, a homogeneous space of a
linear group. This example will play a r\^ole in \S\ref{subsec:norms} below.
\end{example}

\subsection{Fibrations over tori}

The proof of Theorem~\ref{th:main} rests on the fibration method
via the following theorem,
 which is a slightly more precise version of \cite[Théorème~4.2~(ii)]{hwzceh}.
We recall that a variety is \emph{split} if it possesses an irreducible component of multiplicity~$1$
that is geometrically irreducible.

\newcommand{\citezceh}{see \cite[Théorème~4.2~(ii)]{hwzceh}}
\begin{thm}[\citezceh]
\label{th:fibration}
Let~$Q$ be a quasi-trivial torus over a number field~$k$. 
Let~$Z$ be a smooth irreducible variety over~$k$.
Let $\pi:Z \to Q$ be a dominant morphism satisfying the following assumptions:
\begin{enumerate}
\item the geometric generic fibre of~$\pi$ is rationally connected;
\item the fibres of~$\pi$ above the codimension~$1$ points of~$Q$ are split;
\item the morphism $\pi_{\bar k}:Z_{\bar k}\to Q_{\bar k}$ admits a rational section.
\end{enumerate}
For any dense open subset~$U$ of~$Q$ such that~$\pi$ is smooth over~$U$,
and for any Hilbert subset~$H$ of~$U$,
we have the equality 
\begin{align}
Z(k_\Omega)^{\Br_\nr(Z)}=\overline{\bigcup_{q \in U(k) \cap H} Z_q(k_\Omega)^{\Br_{\nr}(Z_q)}}
\end{align}
of subsets of~$Z(k_\Omega)$,
where~$Z_q=\pi^{-1}(q)$.
\end{thm}

Hypothesis~(3) of Theorem~\ref{th:fibration} is stronger than the
hypothesis that appears in \cite[Théorème~4.2~(ii)]{hwzceh}.  The proof, however,
 only depends on the hypothesis formulated in \emph{loc.\ cit.}; we have
opted for stating
the theorem in this way
for the sake of simplicity.
With this minor difference put aside, Theorem~\ref{th:fibration} implies and refines
\cite[Théorème~4.2~(ii)]{hwzceh}.

The  exact argument used in \emph{loc.\ cit.}\ to deduce
\cite[Théorème~4.2~(ii)]{hwzceh} from \cite[Théorème~4.1~(ii)]{hwzceh} also reduces
Theorem~\ref{th:fibration} to the following theorem:

\begin{thm}
\label{th:hararifibration}
Let~$Z$ be a smooth irreducible variety over a number field~$k$,
endowed with a dominant morphism $\pi:Z \to \A^n_k$, for some $n\geq 1$, such that
\begin{enumerate}
\item the geometric generic fibre of~$\pi$ is rationally connected;
\item the fibres of~$\pi$ above the codimension~$1$ points of~$\A^n_k$ are split.
\end{enumerate}
For any dense open subset~$U$ of~$\A^n_k$ such that~$\pi$ is smooth over~$U$,
and for any Hilbert subset~$H$ of~$U$,
we have the equality 
\begin{align}
\label{eq:conclusionfibharari}
Z(k_\Omega)^{\Br_\nr(Z)}=\overline{\bigcup_{q \in U(k) \cap H} Z_q(k_\Omega)^{\Br_{\nr}(Z_q)}}
\end{align}
of subsets of~$Z(k_\Omega)$,
where~$Z_q=\pi^{-1}(q)$.
\end{thm}

In turn, Theorem~\ref{th:hararifibration}
is essentially contained in the work of Harari \cite{harariduke, hararifleches},
though its  statement appears not to have been written down.
We provide a short proof based on the available literature.
We shall use Theorem~\ref{th:hararifibration}
only when $H=U$.
Assuming that $H=U$, however, would not lead to any significant simplification in the proof.

\begin{proof}[Proof of Theorem~\ref{th:hararifibration}]
Assume, first, that $n=1$.
By the theorems of Nagata and Hironaka, the morphism~$\pi$ extends to
a proper morphism $\pi':Z' \to \P^1_k$ for some smooth variety~$Z'$ that
contains~$Z$ as a dense open subset.  Applying \cite[Corollary~4.7,
Remarks~4.8~(i)--(ii), Corollary~6.2~(i)]{hww} to~$\pi'$
now yields~\eqref{eq:conclusionfibharari}.
For $n\geq 2$, we argue by induction.
Fix~$U$ and~$H$ as in the statement of the theorem,
fix a collection of local points $z_\Omega \in Z(k_{\Omega})^{\Br_\nr(Z)}$
and fix a neighbourhood~$\sU$
of~$z_\Omega$ in~$Z(k_\Omega)$.  We need to show the existence of $q \in U(k) \cap H$
such that $\sU \cap Z_q(k_\Omega)^{\Br_{\nr}(Z_q)}\neq\emptyset$.

Let $p:\A^n_k \to \A^1_k$ be the first projection.
For $h \in \A^1_k$, set $Z_h = (p \circ \pi)^{-1}(h)$
and let
 $\pi_h: Z_h \to p^{-1}(h)$ denote the restriction of~$\pi$.
Let $U_0 \subset \A^1_k$ be a dense open subset over which~$p\circ \pi$ is smooth,
small enough that for any $h \in U_0$, the generic fibre
of~$\pi_h$ is rationally connected
(see \cite[Chapter~IV, Theorem~3.5.3]{kollarbook}).
By assumption, there exists a closed subset of codimension~$\geq 2$ in $\A^n_k$ outside of which
the fibres of~$\pi$ are split.  After shrinking~$U_0$, we may assume that the images, by~$p$,
of the irreducible components of this closed subset are either dense in~$\A^1_k$ or disjoint from~$U_0$.
For $h \in U_0$,
the fibres of~$\pi_h$  above the codimension~$1$ points of~$p^{-1}(h)$ are then split.
After further shrinking~$U_0$, we may also assume that $p^{-1}(h) \cap U\neq \emptyset$ for every $h \in U_0$.

By \cite[Lemma~8.12]{hwfibration}, there exists a Hilbert subset $H_0 \subset \A^1_k$
such that for every $h \in U_0(k)\cap H_0$, the set $H \cap p^{-1}(h)$ contains a Hilbert
subset, say~$H_h$, of $p^{-1}(h)=\A^{n-1}_k$.
Let~$\eta$ denote the generic point of~$\A^1_k$.
The generic fibre of~$p\circ \pi$ is endowed with a map to~$\A^{n-1}$ with rationally
connected generic fibre, hence it is itself rationally connected (see \cite[Corollary~1.3]{ghs}).
The closed fibres of~$p \circ\pi$
are each endowed
with a map to~$\A^{n-1}$ whose generic fibre is, by assumption, split; hence they are themselves split.
By the case $n=1$ of Theorem~\ref{th:hararifibration} applied to~$p\circ \pi$,
we deduce the existence of $h \in U_0(k) \cap H_0$ such that
 $\sU \cap Z_h(k_\Omega)^{\Br_{\nr}(Z_h)}\neq\emptyset$.
By the induction hypothesis, we can then apply Theorem~\ref{th:hararifibration} to~$\pi_h$
and finally deduce the existence of $q \in U(k) \cap H_h$ such that
 $\sU \cap Z_q(k_\Omega)^{\Br_{\nr}(Z_q)}\neq\emptyset$.
As $H_h \subset H$, this completes the proof.
\end{proof}

\subsection{Cyclic descent}
\label{subsec:cyclicdescent}

We now establish Theorem~\ref{th:main} in the case
where $\bar G=\Aut(\bar Y/X_{\bar k})$ is a cyclic group.
Since most of the proof works in a slightly greater generality, we only assume,
for now, that~$\bar G$ is an abelian group (and drop the supersolvability assumption on~$\bar Y$).
We shall restrict to the cyclic case
only at the end of~\textsection\ref{subsec:cyclicdescent}.

As~$\bar G$ is abelian, the exact sequence~\eqref{eq:gautgal}
induces
a continuous action of $\Gal(\bar k/k)$
on~$\bar G$.
Let~$G$ be the finite \'etale group scheme
over~$k$ defined by $G(\bar k)=\bar G$.
In the next lemma, the symbol $\Ba(X)$ denotes the subgroup of $\Br_\nr(X)$
consisting of the locally constant classes (i.e.\ the classes
whose image in $\Br(X_{k_v})$ comes from $\Br(k_v)$ for all $v \in \Omega$).

\begin{lem}
\label{lem:existsonef}
If $X(k_{\Omega})^{\Ba(X)}\neq\emptyset$, then $\bar Y \to X_{\bar k}$
descends to a torsor $f:Y \to X$ under~$G$.
\end{lem}

\begin{proof}
Under this assumption, Colliot-Thélène and Sansuc have shown
that the natural map
 $H^2(k,G) \to H^2_\et(X,G)$ is injective
(combine
 \cite[Proposition~2.2.5]{ctsandescent2} for~$X$
with
\cite[Theorem~3.3.1]{wittalb} for a smooth compactification of~$X$;
we note that
in the case of a smooth, proper, rationally connected variety,
the quoted theorem from~\cite{wittalb}
goes back
to~\cite{ctsandescent2}, see \cite[\textsection2.3, Remark]{boctsko}).
By the Hochschild--Serre spectral sequence,
it follows that
the natural map $H^1_\et(X,G) \to H^0(k,H^1_\et(X_{\bar k},G_{\bar k}))$ is surjective.
\end{proof}

\begin{lem}
\label{lem:resolutionquasitrivial}
Over any field, any algebraic group of multiplicative type~$R$ fits into 
a short exact sequence
$1\to R \to T \to Q \to 1$
where~$T$ is a torus and~$Q$ is a quasi-trivial torus.
\end{lem}

\begin{proof}
The character group~$M$ of~$R$ fits into a short exact sequence of Galois modules
$0 \to K \to L \to M \to 0$ with~$K$ and~$L$ torsion-free and finitely generated as
abelian groups; one can even choose~$L$ to be a permutation Galois module.
Doing the same with $\Hom(K,\Z)$ and then dualising,
one finds that~$K$ also fits into an exact sequence of Galois modules
$0 \to K \to P \to C \to 0$ with~$C$ torsion-free and~$P$ permutation.
Let $S=L \oplus_K P$ be the amalgamated sum relative to~$K$.
The exact sequence
$0 \to L \to S \to C \to 0$ shows that~$S$ is torsion-free.
Dualising the short exact sequence $0 \to P \to S \to M \to 0$ therefore provides
the desired resolution.
\end{proof}

Let us fix a torsor
$f:Y\to X$ as in
Lemma~\ref{lem:existsonef}
and a resolution
\begin{align}
\label{eq:resolutionG}
1 \to G \to T \to Q \to 1
\end{align}
given by Lemma~\ref{lem:resolutionquasitrivial} applied to $R=G$.
Let $Z = Y \times_k^G T$ be the contracted product of~$Y$ and~$T$ under~$G$
(i.e.\ the quotient of $Y \times_k T$
 by the action of~$G$
given by $g\cdot (y,t)=(gy,g^{-1}t)$).
Let $g:Z\to X$ and $\pi:Z\to Q$ be the morphisms induced by the two projections.

Let us also fix a collection of local points
$x_\Omega \in X(k_\Omega)^{\Br_{\nr}(X)}$ and a neighbourhood~$\sU$
of~$x_\Omega$ in~$X(k_\Omega)$.  We shall show that
\begin{align}
\label{eq:ufy}
\sU \cap f^\sigma\Big(Y^\sigma(k_\Omega)^{\Br_{\nr}(Y^\sigma)}\Big) \neq\emptyset
\end{align}
for some $[\sigma] \in H^1(k,G)$.
By Remark~\ref{rmk:firstremark}~(iii), this will prove the desired equality~\eqref{eq:th:main}.

To this end, we  apply
 \cite[Corollaire~2.2]{hwzceh} to~$g$.
This yields a $[\tau] \in H^1(k,T)$  such that
$\sU\cap g^\tau\Big(Z^\tau(k_\Omega)^{\Br_{\nr}(Z^\tau)}\Big) \neq\emptyset$.
As the torus~$Q$ is quasi-trivial, Hilbert's Theorem~90 implies that $H^1(k,Q)=0$;
the cohomology class~$[\tau]$ can therefore be lifted to some $[\sigma_0] \in H^1(k,G)$
and we may assume that the cocycle~$\sigma_0$ lifts~$\tau$.
After replacing~$Y$
with~$Y^{\sigma_0}$,
which
has the effect of replacing $Z$ with~$Z^\tau$,
we may then assume that $[\tau]$ is the trivial class,
i.e.\ that
$\sU\cap g\Big(Z(k_\Omega)^{\Br_{\nr}(Z)}\Big) \neq\emptyset$.

To conclude the proof that~\eqref{eq:ufy} holds for some
 $[\sigma] \in H^1(k,G)$,
we shall now
 exploit the structure of a fibration over a quasi-trivial
torus given by the morphism $\pi:Z\to Q$.
For any field extension~$k'/k$ and any $q \in Q(k')$,
we let $Z_q = \pi^{-1}(q)$, which we view as a torsor under~$G$, over~$X_{k'}$, via~$g$.
The inverse image~$T_q$ of~$q$ by the projection $T\to Q$ is a torsor under~$G$, over~$k'$,
whose cohomology class~$[\sigma]$ in $H^1(k',G)$ is the image of~$q$ by the boundary
map of the exact sequence~\eqref{eq:resolutionG}.  As $Z = Y \times_k^G T$, we have $Z_q = Y \times_k^G T_q$ and hence~$Z_q$ and~$Y^\sigma$ are isomorphic as torsors under~$G$, over~$X_{k'}$.
Taking for~$k'$ an algebraically closed field extension of~$k(Q)$ and
for~$q$ a geometric generic point of~$Q$, it follows, first, that
the geometric generic fibre of~$\pi$ is isomorphic to a variety
obtained from~$\bar Y$ by an extension of scalars.
By our assumption on~$\bar Y$,
we deduce that
the generic fibre of~$\pi$ is rationally connected.
Taking $k'=k$, it also follows that
\begin{align}
\label{eq:gzqfsys}
g\Big(Z_q(k_\Omega)^{\Br_{\nr}(Z_q)}\Big)=f^\sigma\Big(Y^\sigma(k_\Omega)^{\Br_{\nr}(Y^\sigma)}\Big)
\end{align}
for every $q \in Q(k)$,
where $[\sigma]\in H^1(k,G)$ is the image of~$q$ by the boundary map of~\eqref{eq:resolutionG}.
In view of~\eqref{eq:gzqfsys}, we will be done if we show the equality
\begin{align}
Z(k_\Omega)^{\Br_\nr(Z)}=\overline{\bigcup_{q \in Q(k)} Z_q(k_\Omega)^{\Br_{\nr}(Z_q)}}
\end{align}
of subsets of~$Z(k_\Omega)$.  Thus, the problem that we need to solve has
been reduced to the question of making the fibration method work for
$\pi:Z\to Q$, a fibration over a quasi-trivial torus whose generic fibre is rationally
connected and all of whose fibres are split (even geometrically integral).
When~$\bar G$ is cyclic, 
a positive answer is given
by Theorem~\ref{th:fibration},
thanks to
the next lemma.

\begin{lem}
\label{lem:usescyclic}
If $\bar G$ is a cyclic group, the morphism
 $\pi_{\bar k}:Z_{\bar k}\to Q_{\bar k}$
admits a rational section.
\end{lem}

\begin{proof}
If~$\bar G$ is cyclic, one can fit
the exact sequence~\eqref{eq:resolutionG}
over~$\bar k$ and
the Kummer exact sequence
into a commutative diagram as pictured below:
\begin{align}
\begin{aligned}
\label{eq:cycliccd}
\xymatrix@R=3ex{
1 \ar[r] & G_{\bar k} \ar[r]\ar[d]^\wr & T_{\bar k}  \ar[d]\ar[r] & Q_{\bar k} \ar[d]\ar[r] & 1 \\
1 \ar[r] & \mmu_{n,\bar k} \ar[r] & \Gmbark \ar[r]^{\times n} & \Gmbark \ar[r] & 1\rlap{.}
}
\end{aligned}
\end{align}
As the right-hand side square of~\eqref{eq:cycliccd} is cartesian,
so is the square obtained by applying the functor $Y \times_k^G -$ to it.
Hence $\pi_{\bar k}:Z_{\bar k}\to Q_{\bar k}$ comes, by a dominant base change,
from the morphism $\pi_0:Y \times_k^G \Gmbark \to \Gmbark$ induced by
the multiplication by~$n$ map $\Gmbark \to \Gmbark$.
In particular, it suffices to check that~$\pi_0$ admits a rational section.
Now, the proper models of the geometric generic fibre of~$\pi_0$ are rationally connected,
since~$\pi_{\bar k}$ and~$\pi_0$ have the same geometric generic fibre.
As the target of~$\pi_0$ is a curve over an algebraically closed field
of characteristic~$0$,
the Graber--Harris--Starr theorem \cite[Theorem~1.1]{ghs},
combined with \cite[Chapter~IV, Theorem~6.10]{kollarbook},
does imply the existence of a rational section of~$\pi_0$.
\end{proof}


\subsection{Proof of Theorem~\ref{th:main} in the general case}
\label{subsec:general-case}

We assume that~$\bar Y$ is a supersolvable finite descent type on~$X$ of class~$n$
and argue by induction on~$n$.  If $n=0$, there is nothing to prove.
Assume that $n>0$ and that the statement of Theorem~\ref{th:main} holds for supersolvable finite
descent types of class $n-1$.

Let $\bar G = \Aut(\bar Y/X_{\bar k})$ and let
 $\{1\}=\bar G_0\subseteq \bar G_1 \subseteq \dots \subseteq \bar G_n = \bar G$
be a filtration satisfying the requirements of Definition~\ref{def:supersolvable}.
The subgroup $\bar G_{n-1}$ of $\Aut(\bar Y/X)$ is normal
since it is stabilised
by the outer action of $\Gal(\bar k/k)$ on~$\bar G$
induced by~\eqref{eq:gautgal}.
Thus $\bar Y'=\bar Y/\bar G_{n-1}$ is 
a finite descent type on~$X$.
As $\Aut(\bar Y'/X_{\bar k})=\bar G_n/\bar G_{n-1}$ is
cyclic, we can apply the case of Theorem~\ref{th:main} already established
in~\textsection\ref{subsec:cyclicdescent}, and deduce that
\begin{align}
\label{eq:xkfp}
X(k_{\Omega})^{\Br_{\nr}(X)}= \overline{\bigcup_{f': Y' \to X} f\Big(Y'(k_{\Omega})^{\Br_{\nr}(Y')}\Big)}\rlap,
\end{align}
where $f':Y'\to X$ ranges over
the isomorphism classes of torsors $Y' \to X$ of type~$\bar Y'$.

\begin{lem}
\label{lem:ypsupersolvable}
Let
 $f':Y' \to X$ be a torsor  of type~$\bar Y'$
and
 $\iota:Y'_{\bar k}\isoto \bar Y'$
be an isomorphism
of $X_{\bar k}$\nobreakdash-schemes.
Viewing the scheme~$\bar Y$ as
a $Y'_{\bar k}$\nobreakdash-scheme via~$\iota$,
it is a
supersolvable finite descent type on~$Y'$ of class $n-1$.
\end{lem}

\begin{proof}
As~$\bar Y$ is Galois over~$X$, it is Galois over~$Y'$.
Moreover, the commutative diagram
\begin{align*}
\xymatrix@R=1.7ex{
1 \ar[r] & \vphantom{\Aut(\bar Y/X)}\bar G \ar[r] & \vphantom{\bar G}\Aut(\bar Y/X) \ar[r] & \Gal(\bar k/k) \ar[r] & 1 \\
1 \ar[r] & \vphantom{\Aut(\bar Y/Y')}\bar G_{n-1} \ar[r]
\ar@{}[u]|-*{\cup}
 & \Aut(\bar Y/Y')\vphantom{\bar G_{n-1}}
\ar@{}[u]|-*{\cup}
\ar[r] & \Gal(\bar k/k) \ar@{=}[u] \ar[r] & 1
}
\end{align*}
shows that the outer action of $\Gal(\bar k/k)$ on~$\bar G_{n-1}$ coming from the bottom row
stabilises the subgroups $\bar G_1,\dots,\bar G_{n-2}$ of~$\bar G_{n-1}$,
since these subgroups
are stable under the outer action of $\Gal(\bar k/k)$
on~$\bar G$ coming from the top row.
\end{proof}

For any torsor $f':Y'\to X$ of type~$\bar Y'$ and for any  $X_{\bar k}$\nobreakdash-isomorphism
  $\iota:Y'_{\bar k}\isoto \bar Y'$,
Lemma~\ref{lem:ypsupersolvable} and the induction hypothesis imply the equality
\begin{align}
\label{eq:ypfs}
Y'(k_{\Omega})^{\Br_{\nr}(Y')}= \overline{\bigcup_{f'': Y \to Y'} f''\Big(Y(k_{\Omega})^{\Br_{\nr}(Y)}\Big)}
\end{align}
of subsets of~$Y'(k_\Omega)$,
where $f'':Y\to Y'$ ranges over
the isomorphism classes of torsors $Y \to Y'$ of type~$\bar Y$ (viewing~$\bar Y$ as an
$Y'_{\bar k}$\nobreakdash-scheme via~$\iota$).  Now for any such~$f'$, $\iota$ and~$f''$,
the composition $f'\circ f'':Y\to X$ is a torsor of type~$\bar Y$.  Hence combining~\eqref{eq:xkfp}
with~\eqref{eq:ypfs}
yields~\eqref{eq:th:main}.
This completes the proof of Theorem~\ref{th:main}.

\section{Applications}
\label{sec:applications}

We now discuss applications of supersolvable
descent to rational points on homogeneous spaces and to Galois theory,
pursuing and expanding the investigations of~\cite{hwzceh}.
Unless otherwise noted, the field~$k$ will be assumed in~\textsection\ref{sec:applications}
to be a number field.

\subsection{Homogeneous spaces of linear algebraic groups}
\label{subsec:homogeneous-spaces}

In Theorem~\ref{th:genthb} below,
we apply supersolvable descent to
the validity of Conjecture~\ref{conj:ct} for homogeneous spaces of linear algebraic groups.
Additional notation and terminology
that is useful for dealing with
stabilisers of geometric points on such homogeneous spaces
will first be introduced in~\textsection\ref{subsubsec:sigmaalgebraic}.
Theorem~\ref{th:genthb} is stated in~\textsection\ref{subsubsec:statement}
and proved in~\textsection\ref{subsubsec:proofthgenb}.

\subsubsection{Outer Galois actions and $\sigma$-algebraic maps}
\label{subsubsec:sigmaalgebraic}

We first
introduce
 \emph{$\sigma$\nobreakdash-algebraic} maps,
following Borovoi \cite[\textsection1.1]{borovoiab}.

\begin{defn}
Given a field automorphism~$\sigma$ of~$\bar k$,
a \emph{$\sigma$\nobreakdash-algebraic map} between two varieties~$V$, $W$ over~$\bar k$
is a
morphism of schemes $f: V \to W$
that makes the square
\begin{align*}
\xymatrix@R=4ex{
V \ar[rr]^{f}\ar[d]^(.45){\varepsilon_V} && W \ar[d]^(.45){\varepsilon_W} \\
\Spec(\bar k) \ar[rr]^{\Spec\big(\sigma^{-1}\big)} && \Spec(\bar k)
}
\end{align*}
commute,
where~$\varepsilon_V$ and~$\varepsilon_W$ are the structure morphisms of~$V$ and~$W$.
Equivalently,
 if $\sigma^*W$ denotes the variety over~$\bar k$ with underlying scheme~$W$
and structure morphism $\Spec(\sigma)\circ \varepsilon_W$, a $\sigma$\nobreakdash-algebraic map $f:V\to W$ is a morphism of varieties $V \to \sigma^*W$.
\end{defn}

A $\sigma$\nobreakdash-algebraic map is generally not a morphism of varieties.
Nonetheless, any $\sigma$\nobreakdash-algebraic map $f:V\to W$ induces a
map  $f_*:V(\bar k) \to W(\bar k)$, since the sets~$V(\bar k)$ and~$W(\bar k)$
 can be identified with the sets of closed points of the schemes~$V$ and~$W$.
We shall say that a map $V(\bar k)\to W(\bar k)$ is \emph{$\sigma$\nobreakdash-algebraic}
if it coincides with~$f_*$ for a $\sigma$\nobreakdash-algebraic map $f:V\to W$,
and that it is \emph{algebraic} if it is $\sigma$\nobreakdash-algebraic with $\sigma=\Id_{\bar k}$.

\begin{rmks}
\label{rmks:sigmaalg}
(i) If~$V$ and~$W$ are non-empty varieties over~$\bar k$, a morphism of schemes $f:V \to W$ can be a $\sigma$\nobreakdash-algebraic map for at most one automorphism~$\sigma$ of~$\bar k$.

(ii)
As $\Spec\big(\tau^{-1}\big) \circ \Spec\big(\sigma^{-1}\big) = \Spec\big(\sigma^{-1}\tau^{-1}\big)
= \Spec\big((\tau\sigma)^{-1}\big)$,
precomposing a $\tau$\nobreakdash-algebraic map with
a $\sigma$\nobreakdash-algebraic map yields a $\tau\sigma$\nobreakdash-algebraic map.
In particular, the class of $\sigma$\nobreakdash-algebraic maps is closed under composition with
algebraic maps.

(iii)
If $V = V_0 \times_k \Spec(\bar k)$ for a variety~$V_0$ over~$k$,
then for any $\sigma \in \Gal(\bar k/k)$,
the morphism of schemes $f:V \to V$ given by
$\Id_{V_0} \times_k \Spec\big(\sigma^{-1}\big)$
is a $\sigma$\nobreakdash-algebraic map.
The map $f_*:V(\bar k)\to V(\bar k)$ that it induces is $v \mapsto \sigma(v)$.

(iv)
Let $V = V_0 \times_k \Spec(\bar k)$ for a variety~$V_0$ over~$k$
and~$U$ be an irreducible finite étale $V$\nobreakdash-scheme.
Let $\rho:\Aut(U/V_0) \to \Gal(\bar k/k)$ denote the natural map,
which factors through $\Aut(V/V_0)=\Gal(\bar k/k)$.
Then $a:U\to U$ is a $\rho(a)$\nobreakdash-algebraic map
for any $a \in \Aut(U/V_0)$.
\end{rmks}

We recall that if~$X$ is a (left) homogeneous space of a connected linear algebraic
group~$L$ over~$k$ and if
 $H_{\bar x} \subset L_{\bar k}$
denotes
the stabiliser
 of a point $\bar x \in X(\bar k)$,
viewed as an algebraic group over~$\bar k$, the exact sequence
\begin{align}
\label{eq:hegal}
1 \to H_{\bar x}(\bar k) \to G_{\bar x}  \to \Gal(\bar k/k) \to 1\rlap,
\end{align}
where $G_{\bar x} = \big\{ (\ell, \sigma) \in L(\bar k) \rtimes \Gal(\bar k/k)
\mkern2mu;\mkern1.5mu
\ell\sigma(\bar x)=\bar x\big\}$,
induces
a continuous
outer action of the profinite group $\Gal(\bar k/k)$
on the discrete group $H_{\bar x}(\bar k)$
(see \cite[\textsection2.3]{demarchelucchinireduction}).
Thus, the discrete group $H_{\bar x}(\bar k)$ receives a continuous outer  action of~$\Gal(\bar k/k)$
while the algebraic group~$H_{\bar x}$ is only defined over~$\bar k$.
The notion of $\sigma$\nobreakdash-algebraic map
allows one to
 reconcile this outer action with the algebraic structure of~$H_{\bar x}$, as shown by the following proposition.

\begin{prop}
\label{prop:outeractionissigmaalg}
Let $\sigma \in \Gal(\bar k/k)$.
Any group automorphism of~$H_{\bar x}(\bar k)$
that represents
 the outer action of $\sigma$
is induced by a $\sigma$\nobreakdash-algebraic map $H_{\bar x}\to H_{\bar x}$.
\end{prop}

\begin{proof}
Let $(\ell,\sigma) \in G_{\bar x}$.
The automorphism $m \mapsto \ell \sigma(m) \ell^{-1}$ of~$L(\bar k)$,
being the composition of the $\sigma$\nobreakdash-algebraic map $m \mapsto \sigma(m)$
with the algebraic map $m \mapsto \ell m \ell^{-1}$, is itself $\sigma$\nobreakdash-algebraic
(see
Remarks~\ref{rmks:sigmaalg}~(ii)--(iii)), i.e.\ it equals~$f_*$
for
a $\sigma$\nobreakdash-algebraic map $f:L_{\bar k} \to L_{\bar k}$.
As~$f_*$ stabilises $H_{\bar x}(\bar k)$
and as~$H_{\bar x}$ is a reduced closed subscheme of~$L_{\bar k}$,
the scheme morphism~$f$ stabilises~$H_{\bar x}$.
As the resulting scheme morphism $g:H_{\bar x} \to H_{\bar x}$
is a $\sigma$\nobreakdash-algebraic map
and as the automorphism $g_*$ coincides with conjugation by~$(\ell,\sigma)$,
the proposition is proved.
\end{proof}

\begin{cor}
\label{cor:hx0stable}
Let~$H_{\bar x}^0$ denote the connected component of the identity in~$H_{\bar x}$.
The outer action of~$\Gal(\bar k/k)$ on~$H_{\bar x}(\bar k)$ induced by~\eqref{eq:hegal}
stabilises $H_{\bar x}^0(\bar k)$ and hence
induces an outer action
of~$\Gal(\bar k/k)$ on the finite group $\pi_0(H_{\bar x})$.
\end{cor}

\begin{proof}
This follows from Proposition~\ref{prop:outeractionissigmaalg},
as any scheme morphism $H_{\bar x} \to H_{\bar x}$ that preserves the identity point
must stabilise the open subscheme~$H_{\bar x}^0$.
\end{proof}

\subsubsection{Statement}
\label{subsubsec:statement}

We now formulate Theorem~\ref{th:genthb}, our main application of supersolvable descent to
homogeneous spaces of linear algebraic groups, and discuss its first consequences.

\begin{thm}
\label{th:genthb}
Let~$X$ be a homogeneous space of
 a connected linear algebraic group~$L$ over a number field~$k$.
Let $\bar x \in X(\bar k)$.
Let $H_{\bar x}$ denote the stabiliser of~$\bar x$ and $N \subset H_{\bar x}$
be a normal algebraic subgroup of finite index satisfying the following two assumptions:
\begin{enumerate}
\item
the outer action of $\Gal(\bar k/k)$ on~$H_{\bar x}(\bar k)$
induced by~\eqref{eq:hegal} stabilises~$N(\bar k)$;
\item the quotient $H_{\bar x}(\bar k)/N(\bar k)$
is supersolvable
in the sense of Definition~\ref{def:supersolvable},
with respect to the outer
action of $\Gal(\bar k/k)$ on $H_{\bar x}(\bar k)/N(\bar k)$ induced by~\eqref{eq:hegal}.
\end{enumerate}
Let~$Y$ range over the
homogeneous spaces of~$L$ over~$k$ that satisfy the following condition:
\vspace*{1.5pt}
\begin{flushright}
$(\star)\mkern9mu$\begin{minipage}[t]{.92\textwidth}
there exist
 an $L$\nobreakdash-equivariant map $Y \to X$
and a lifting $\bar y \in Y(\bar k)$ of~$\bar x$ whose stabiliser,
as an algebraic subgroup of~$L_{\bar k}$, is equal to~$N$.
\end{minipage}
\end{flushright}
\vspace*{3pt}
If
Conjecture~\ref{conj:ct} (resp.\ the implication $Y(k_\Omega)^{\Br_{\nr}(Y)}\neq\emptyset \Rightarrow Y(k)\neq\emptyset$)
holds for all such~$Y$,
then Conjecture~\ref{conj:ct} holds for~$X$
(resp.\ then
 $X(k_\Omega)^{\Br_{\nr}(X)}\neq\emptyset \Rightarrow X(k)\neq\emptyset$).
\end{thm}

\begin{rmk}
\label{rmk:weakerstatement}
The weaker statement
 obtained by allowing~$Y$ to
range over all homogeneous spaces of~$L$ over~$k$ whose geometric stabilisers
are isomorphic to~$N$ as algebraic groups over~$\bar k$ is sufficient
for the applications
of Theorem~\ref{th:genthb}
 considered in this article.
\end{rmk}

When~$N = H^0_{\bar x}$, the first hypothesis of Theorem~\ref{th:genthb} is satisfied,
by Corollary~\ref{cor:hx0stable}. On the other hand,
Conjecture~\ref{conj:ct} holds for homogeneous
spaces of~$L$ with connected geometric stabilisers, by a theorem of Borovoi (see \cite[Corollary~2.5]{borovoi}).
Thus, we deduce:

\begin{cor}
\label{cor:genthb}
Let~$X$ be a homogeneous space of
 a connected linear algebraic group~$L$ over a number field~$k$.
Let $\bar x \in X(\bar k)$.
Assume that the group of connected components of the stabiliser of~$\bar x$
is supersolvable
in the sense of Definition~\ref{def:supersolvable},
with respect to the outer action of $\Gal(\bar k/k)$
given by
Corollary~\ref{cor:hx0stable}.
Then Conjecture~\ref{conj:ct} holds for $X$. 
\end{cor}

Corollary~\ref{cor:genthb} simultaneously generalises
Borovoi's theorem mentioned above
(where the geometric stabilisers are connected)
and \cite[Théorème~B]{hwzceh} (where the geometric stabilisers are finite and supersolvable).
In fact, even in the particular case of finite and supersolvable geometric 
stabilisers,
Corollary~\ref{cor:genthb} strictly generalises \cite[Théorème~B]{hwzceh},
as it relaxes all hypotheses on the ambient linear group~$L$,
assumed
in \emph{loc.\ cit.}\ to be semi-simple and simply connected.
What is more,  when~$L$ is semi-simple and simply connected,
Corollary~\ref{cor:genthb} can be used to ensure the validity of
Conjecture~\ref{conj:ct} even in cases where the geometric stabilisers are not supersolvable,
as the following example shows.

\begin{example}
Assume that~$L$ is semi-simple and simply connected.
Then, by a theorem of Borovoi,
Conjecture~\ref{conj:ct}
holds for all~$Y$ as in Theorem~\ref{th:genthb} if~$N$ is abelian
 (see \cite[Corollary~2.5]{borovoi}).
Thus, Theorem~\ref{th:genthb} implies the validity of Conjecture~\ref{conj:ct} for any homogeneous
space of~$L$ whose geometric stabilisers are extensions of a supersolvable finite group by an abelian algebraic subgroup
(compatibly with the outer Galois action, as stated in Theorem~\ref{th:genthb}~(1)--(2)).
\end{example}

Combining
Theorem~\ref{th:genthb}
with the work of Neukirch~\cite{neukirch-solvable}
also yields Conjecture~\ref{conj:ct} for homogeneous spaces of~$\SL_n$ whose geometric stabilisers
can be written, compatibly with the outer action of~$\Gal(\bar k/k)$, as extensions
of a supersolvable finite group by a solvable finite group whose order is
coprime to the number of roots of unity in~$k$.

Non-solvable examples where Theorem~\ref{th:genthb} can be applied
will be discussed in \S\ref{subsec:non-solvable}.

\subsubsection{Proof of Theorem~\ref{th:genthb}} 
\label{subsubsec:proofthgenb}

Set $\bar Y=L_{\bar k}/N$. We view~$\bar Y$ as an $X_{\bar k}$\nobreakdash-scheme
through the projection
\begin{align}
\label{eq:lkbarnxkbar}
\bar Y = L_{\bar k}/N \to L_{\bar k}/H_{\bar x}=X_{\bar k}\rlap.
\end{align}
This projection is a torsor under $H_{\bar x}/N$, so that there is a natural short exact sequence
\begin{align}
\label{eq:autbaryxbark}
1 \to N(\bar k) \to H_{\bar x}(\bar k) \xrightarrow{\phi} \Aut(\bar Y/X_{\bar k}) \to 1\rlap.
\end{align}
Explicitly, the map~$\phi$ sends any $\ell \in H_{\bar x}(\bar k)$
to the automorphism
 of the variety~$\bar Y$ over~$\bar k$
which on $\bar k$\nobreakdash-points, i.e.\ on the quotient set $L(\bar k)/N(\bar k)$,
is given by $mN(\bar k) \mapsto m\ell^{-1}N(\bar k)$.

For the statement of the next lemma, we recall that $N(\bar k)$
is a normal subgroup of the middle term~$G_{\bar x}$ of~\eqref{eq:hegal},
as a consequence of assumption~(1) of Theorem~\ref{th:genthb}.

\begin{lem}
\label{lem:lkbarn is a finite descent type}
The $X_{\bar k}$\nobreakdash-scheme $\bar Y$ is a finite descent type on~$X$.
In addition, the short exact sequence~\eqref{eq:gautgal}
can be identified with the sequence obtained
from~\eqref{eq:hegal} by replacing the first two terms of~\eqref{eq:hegal}
with their quotients by the normal subgroup~$N(\bar k)$.
\end{lem}

\begin{proof}
Let
 $\sigma \in \Gal(\bar k/k)$
and $\ell \in L(\bar k)$ be such that $\ell\sigma(\bar x)=\bar x$.
By assumption~(1) of Theorem~\ref{th:genthb},
the automorphism $m \mapsto \ell \sigma(m) \ell^{-1}$ of~$L(\bar k)$
stabilises the subgroup $N(\bar k)$.
We deduce that the $\sigma$\nobreakdash-algebraic map
$L(\bar k)\to L(\bar k)$, $m \mapsto \sigma(m) \ell^{-1}$
induces a $\sigma$\nobreakdash-algebraic map
 $\bar Y(\bar k) \to \bar Y(\bar k)$.
The latter is the top horizontal arrow of a
commutative square
\begin{align}
\begin{aligned}
\label{eq:firstsquare proof finite descent type}
\xymatrix@R=3ex{
\bar Y(\bar k) \ar[r] \ar[d] & \bar Y(\bar k) \ar[d] \\
X(\bar k) \ar[r] & X(\bar k)
}
\end{aligned}
\end{align}
whose lower horizontal arrow
is the $\sigma$\nobreakdash-algebraic map
$m \mapsto \sigma(m)$
and whose vertical arrows are given by $m \mapsto m\bar x$
(i.e.\ are induced by~\eqref{eq:lkbarnxkbar}).
As the horizontal arrows are $\sigma$\nobreakdash-algebraic and the vertical
ones are algebraic, the square~\eqref{eq:firstsquare proof finite descent type}
is induced on closed points by a commutative square of schemes
\begin{align}
\begin{aligned}
\label{eq:diag proof finite descent type}
\xymatrix@C=8em@R=3ex{
\bar Y \ar[r] \ar[d] & \bar Y \ar[d] \\
X_{\bar k} \ar[r]^{\Id_X \times_k \Spec\big(\sigma^{-1}\big)} & X_{\bar k}
}
\end{aligned}
\end{align}
whose horizontal arrows are $\sigma$-algebraic maps and both of whose vertical arrows are the projection~\eqref{eq:lkbarnxkbar}.
Hence the top horizontal arrow is an automorphism
of the $X$\nobreakdash-scheme $\bar Y$.

With~$G_{\bar x}$ as in~\eqref{eq:hegal},
let $\psi:G_{\bar x} \to \Aut(\bar Y/X)$ denote the map
that sends $(\ell,\sigma)$ to this $X$\nobreakdash-scheme automorphism of~$\bar Y$
and let $\rho:\Aut(\bar Y/X) \to \Aut(X_{\bar k}/X)=\Gal(\bar k/k)$
denote the natural morphism.
One readily checks that~$\psi$ is a homomorphism
and that the
 exact sequence~\eqref{eq:hegal}
fits into a commutative
 diagram
\begin{align*}
\xymatrix@R=3ex@C=2em{
1 \ar[r] & H_{\bar x}(\bar k) \ar[r] \ar[d]^(.45)\phi & G_{\bar x} \ar[d]^(.45)\psi
\ar[r] & \Gal(\bar k/k)
 \ar@{=}[d] \ar[r] & 1 \\
1 \ar[r] & \Aut(\bar Y/X_{\bar k}) \ar[r] &
\Aut(\bar Y/X)
\ar[r]^(.51)\rho & \Gal(\bar k/k)\rlap,
}
\end{align*}
where~$\phi$ comes from~\eqref{eq:autbaryxbark}.
(The commutativity of the left square follows from the explicit descriptions of~$\psi$ and~$\phi$;
that of the right square follows from Remarks~\ref{rmks:sigmaalg}~(iv) and~(i).)
This diagram shows that~$\rho$ is surjective, so that~$\bar Y$ is indeed a
finite descent type on~$X$.
In addition, it follows from this diagram
and from the exact sequence~\eqref{eq:autbaryxbark}
that the exact sequence~\eqref{eq:gautgal}
can be identified
as indicated in the statement of the lemma.
\end{proof}

\begin{lem}
\label{lem:yishomogeneous}
Let $Y \to X$ be a torsor of type~$\bar Y$.
There exists a unique action of~$L$ on~$Y$ such that the morphism $Y\to X$
is $L$\nobreakdash-equivariant.
With respect to this action, the variety~$Y$ is a homogeneous space of~$L$,
and there exists a lifting $\bar y \in Y(\bar k)$ of~$\bar x$ whose stabiliser,
as an algebraic subgroup of~$L_{\bar k}$, is equal to~$N$.
\end{lem}

\begin{proof}
This lemma is valid over any field~$k$ of characteristic~$0$.
In order to prove it, we may and will assume that $k=\bar k$.
Indeed, by Galois descent, the existence of the action of~$L$ on~$Y$
follows from its existence and unicity over~$\bar k$; and all other
conclusions of the lemma are of a geometric nature.
Let us write $X=L/H_{\bar x}$ and fix an $X$\nobreakdash-isomorphism $Y \simeq \bar Y=L/N$.
The existence of an action of~$L$ on~$Y$ satisfying all of the conclusions
of the lemma is now obvious, and we need only check its unicity.
For the latter,
the first paragraph of the proof of
 \cite[Proposition~5.1]{hwzceh} applies verbatim
(and it does not depend on the
hypotheses of semi-simplicity and simple connectedness made in \emph{loc.\ cit.}).
\end{proof}

By Lemma~\ref{lem:lkbarn is a finite descent type}
and by assumption~(2) of Theorem~\ref{th:genthb},
the $X_{\bar k}$\nobreakdash-scheme~$\bar Y$ is a
supersolvable finite descent type on~$X$.
Moreover, Lemma~\ref{lem:yishomogeneous}
and the final assumption of Theorem~\ref{th:genthb}
ensure that for any torsor $Y \to X$ of type~$\bar Y$,
the set $Y(k)$ is a dense subset of $Y(k_\Omega)^{\Br_{\nr}(Y)}$
(resp.\ the implication
$Y(k_\Omega)^{\Br_{\nr}(Y)}\neq\emptyset \Rightarrow Y(k)\neq\emptyset$
holds).
Theorem~\ref{th:main} now implies
that the set $X(k)$ is a dense subset of $X(k_\Omega)^{\Br_{\nr}(X)}$
(resp.\ that
$X(k_\Omega)^{\Br_{\nr}(X)}\neq\emptyset \Rightarrow X(k)\neq\emptyset$).
Thus, Theorem~\ref{th:genthb} is proved.

\subsubsection{A special case: homogeneous spaces of~$\SL_n$ with finite stabilisers}
\label{subsec:non-solvable}

We now spell out a useful corollary of Theorem~\ref{th:genthb} in the special case where $L=\SL_n$
(Corollary~\ref{cor:almost-complete} below).
We shall apply it
in~\textsection\ref{subsec:inversegalois}
to the inverse Galois problem and to the Grunwald problem.

To prepare for the statement of Corollary~\ref{cor:almost-complete}, let us
recall that a finite group is said to be \emph{complete} if its centre is trivial
and all its automorphisms are inner.
We shall say that a finite group~$N$ is \emph{almost complete} if its centre is trivial
and the homomorphism $\Aut(N)\to \Out(N)$ admits a section.

\begin{cor}
\label{cor:almost-complete}
Let $G$ be a finite group equipped with an outer action of~$\Gal(\bar k/k)$.
Let~$X$ be a homogeneous space of~$\SL_n$ over a number field~$k$,
with geometric stabilisers isomorphic to~$G$ as groups endowed with an outer action of~$\Gal(\bar k/k)$.
Let $N \subseteq G$ be a normal subgroup stable under the outer action of~$\Gal(\bar k/k)$.
Assume that the  group $G/N$ is supersolvable,
in the sense of Definition~\ref{def:supersolvable},
with respect to the induced outer action of~$\Gal(\bar k/k)$. Then:
\begin{enumerate}
\item[(i)]
If the finite group~$N$ is almost complete, then $X(k_{\Omega})^{\Br_{\nr}(X)} \neq \emptyset \Rightarrow X(k) \neq \emptyset$.
\item[(ii)]
If the finite group~$N$ is almost complete and if,
for any
 finite étale subgroup scheme~$\tilde N$ of~$\SL_n$ over~$k$ such that the groups $\tilde N(\bar k)$
and~$N$ are isomorphic,
the weak approximation property holds for the quotient variety
$\SL_n/\tilde N$,
then Conjecture~\ref{conj:ct} holds for~$X$.
\item[(iii)]
If the finite group~$N$ is complete and if,
for any embedding $N \hookrightarrow \SL_n(k)$,
letting~$\tilde N$ denote the constant subgroup scheme of~$\SL_n$ with $\tilde N(k)=N$,
the weak approximation property holds for the quotient variety
$\SL_n/\tilde N$,
then Conjecture~\ref{conj:ct} holds for~$X$.
\end{enumerate}
\end{cor}

In~(ii) of Corollary~\ref{cor:almost-complete},
we do not require any compatibility between the Galois action on~$\tilde N(\bar k)$
and the given outer Galois action on~$G$.
  One could obtain a slightly
more precise statement by doing so (see Remark~\ref{rmk:weakerstatement}).

\begin{proof}[Proof of Corollary~\ref{cor:almost-complete}]
By Theorem~\ref{th:genthb},
it is enough to prove that for any homogeneous space~$Y$ of~$\SL_n$ over~$k$ whose geometric stabilisers
are isomorphic, as abstract groups, to~$N$,
Conjecture~\ref{conj:ct} holds for~$Y$ (resp.\ the implication
$Y(k_\Omega)^{\Br_{\nr}(Y)}\neq\emptyset \Rightarrow Y(k)\neq\emptyset$)
holds) if the assumptions of~(ii)--(iii) (resp.\ of~(i)) are satisfied.
We shall see that~$Y$ even satisfies the weak approximation property (resp.\ that $Y(k)\neq\emptyset$
unconditionally).

Let us fix a point $\bar y \in Y(\bar k)$ and a group isomorphism $H_{\bar y}(\bar k)\simeq N$,
and consider the resulting exact sequence
\begin{align}
\label{eq:hegaly}
1 \to N \to G_{\bar y}  \to \Gal(\bar k/k) \to 1\rlap,
\end{align}
where $G_{\bar y} = \big\{ (\ell, \sigma) \in \SL_n(\bar k) \rtimes \Gal(\bar k/k)
\mkern2mu;\mkern1.5mu \ell\sigma(\bar y)=\bar y\big\}$.
We endow $\SL_n(\bar k) \rtimes \Gal(\bar k/k)$ with the product of the discrete topology
on~$\SL_n(\bar k)$ and the Krull topology on~$\Gal(\bar k/k)$, and~$G_{\bar y}$ with the induced topology,
so that~\eqref{eq:hegaly} becomes an exact sequence of profinite groups.
By Lemma~\ref{lem:observation} below and by the hypothesis that~$N$ is almost complete,
this sequence admits a continuous homomorphic splitting $s:\Gal(\bar k/k)\to G_{\bar y}$
and we can assume that the image of~$s$ commutes with $N \subset G_{\bar y}$ if in addition~$N$ is complete.
Composing~$s$ with the projection $G_{\bar y}\to \SL_n(\bar k)$ yields a continuous
cocycle $\Gal(\bar k/k)\to \SL_n(\bar k)$.  As the Galois cohomology set $H^1(k,\SL_n)$ is a singleton
(Hilbert's Theorem~90), there exists $b \in \SL_n(\bar k)$
such that $s(\sigma)=(b^{-1}\sigma(b), \sigma)$ for all $\sigma \in \Gal(\bar k/k)$.
The very definition of~$G_{\bar y}$ now shows that $\sigma(b\bar y)=b\bar y$ for all $\sigma \in \Gal(\bar k/k)$,
in other words $b\bar y \in Y(k)$.  This already proves that $Y(k)\neq \emptyset$ and hence takes care of
Corollary~\ref{cor:almost-complete}~(i).
Let us denote by $\tilde N \subset \SL_n$ the stabiliser of the rational point $b\bar y \in Y(k)$, so that $Y=\SL_n/\tilde N$.
As $\tilde N(\bar k)=b H_{\bar y}(\bar k) b^{-1}$, the groups $\tilde N(\bar k)$ and~$N$ are isomorphic,
and
Corollary~\ref{cor:almost-complete}~(ii) follows.  Finally, the condition that the image of~$s$ commutes with~$N$ is equivalent to
 $\sigma(bhb^{-1})=bhb^{-1}$ for all $\sigma \in \Gal(\bar k/k)$ and all $h \in H_{\bar y}(\bar k)$;
therefore this condition implies that $\tilde N$ is a constant group scheme over~$k$, and
Corollary~\ref{cor:almost-complete}~(iii) is proved.
\end{proof}

\begin{lem}\label{lem:observation}
Let~$N$ be a finite group with trivial centre. Then
\begin{enumerate}
\item
$N$ is almost complete if and only if every short exact sequence of profinite groups
\begin{align*}
1 \to N \to G \to H \to 1
\end{align*}
splits as a semi-direct product of profinite groups $G \cong N \rtimes H$;
\item
$N$ is complete if and only if every short exact sequence of profinite groups
\[ 1 \to N \to G \to H \to 1 \]
splits as a direct product of profinite groups $G \cong N \times H$. 
\end{enumerate}
\end{lem}

\begin{proof}
As the centre of~$N$ is trivial, the group of inner automorphisms of~$N$ can be identified with~$N$ and
we have a short exact sequence of finite groups
\begin{align}
\label{eq:nautnoutn}
1 \to N \to \Aut(N) \to \Out(N) \to 1\rlap.
\end{align}
If~\eqref{eq:nautnoutn} splits as a semi-direct product, then~$N$ is almost complete, by definition.
If~\eqref{eq:nautnoutn}
splits as a direct product, then the outer action of~$\Out(N)$ on~$N$ induced by~\eqref{eq:nautnoutn}
is trivial.  As this outer action coincides with the canonical outer action of~$\Out(N)$ on~$N$,
it follows that $\Out(N)$ is trivial, i.e.\ $N$ is complete.

Conversely, let us assume that~$N$ is almost complete (resp.\ complete).  Any short exact sequence
as in the statement of the lemma canonically fits into a commutative diagram
\begin{align*}
\xymatrix@R=3ex{
1 \ar[r] & N \ar[r]\ar@{=}[d] & G \ar[r]\ar[d] & H \ar[d]\ar[r] & 1 
\\
1 \ar[r] & N \ar[r] & \Aut(N) \ar[r] & \Out(N) \ar[r] & 1\rlap,
}
\end{align*}
where the middle vertical arrow sends $g \in G$ to the automorphism $z\mapsto gzg^{-1}$ of~$N$.
Thus, the upper row is obtained by pull-back from the lower row, and
the upper row splits as a semi-direct (resp.\ direct)
product if so does the lower row.
\end{proof}


A full characterisation of almost complete simple groups is established
in~\cite{lmm03}. These include for example all the simple alternating
groups $A_n$ for $n \neq 6$, all the sporadic simple groups, and all
Chevalley groups $L(\F_p)/Z(L(\F_p))$ where~$p\geq 5$ is a prime and~$L$ is
a split simple simply connected algebraic group over~$\F_p$ (see
\cite{lmm03} and \cite[p.~A-14]{borelproperties}).  In addition, if $G$ is
a finite group all of whose composition factors are almost complete simple groups, then
$G$ itself is almost complete, as can be seen by
mimicking the proof of~\cite[Theorem~2]{anti-solvable}
and exploiting Lemma~\ref{lem:observation} above.

These remarks already provide many examples to which Corollary~\ref{cor:almost-complete}~(i) can be applied.
We now illustrate,
in Corollary~\ref{cor:symmetricalternating},
cases~(ii) and~(iii) of Corollary~\ref{cor:almost-complete}.

\begin{cor}
\label{cor:symmetricalternating}
Let $G$ be a finite group equipped with an outer action of~$\Gal(\bar k/k)$.
Let $N \subseteq G$ be a normal subgroup stable under this outer action.
Assume that the group $G/N$ is supersolvable,
in the sense of Definition~\ref{def:supersolvable},
with respect to the induced outer action of~$\Gal(\bar k/k)$.
Assume that one of the following two conditions holds:
\begin{enumerate}
\item $N$ is isomorphic to the symmetric group~$S_m$ with $m\neq 6$;
\item $N$ is isomorphic to the alternating group~$A_5$.
\end{enumerate}
Then Conjecture~\ref{conj:ct} holds for any homogeneous space of~$\SL_n$
whose geometric stabilisers are isomorphic,
as groups endowed with an
outer action of~$\Gal(\bar k/k)$, to~$G$.
\end{cor}

It should be noted that Corollary~\ref{cor:symmetricalternating} in case~(2) with $G=N$ was
first established by
Boughattas and Neftin~\cite{boughattasneftin}, who gave in this way the first example of a non-abelian simple
group~$N$ such that Conjecture~\ref{conj:ct} holds for
any homogeneous space of~$\SL_n$ with geometric stabilisers isomorphic, as abstract
groups, to~$N$.  We provide an alternative proof, based on the special properties of del Pezzo surfaces of degree~$5$.

\begin{proof}[Proof of Corollary~\ref{cor:symmetricalternating}]
If $N=S_2$, the supersolvability of~$G/N$ implies that of~$G$ itself,
and the conclusion results from
Corollary~\ref{cor:genthb}.
If $N=S_m$ with  $m\notin\{2,6\}$, then~$N$ is a complete
group
for which the Noether problem has a positive answer.
In this case,
 Corollary~\ref{cor:almost-complete}~(iii) can be applied:
the variety $\SL_n/\tilde N$ appearing in its statement is stably rational
and therefore satisfies the weak approximation property.

It only remains to treat the case $N=A_5$, which is an almost complete finite group.
We shall prove that
 Corollary~\ref{cor:almost-complete}~(ii) can be applied,
i.e.\ that
 for any finite étale
subgroup scheme~$\tilde N$ of~$\SL_n$ over~$k$ such that $\tilde N(\bar k)$ is isomorphic to~$A_5$,
 the variety $\SL_n/\tilde N$ satisfies the weak approximation property;
in fact, we shall even prove that $\SL_n/\tilde N$ is stably rational.

Let~$Y$ denote the split del Pezzo surface of degree~$5$ over~$k$, i.e.\ the blow-up of~$\P^2_k$ along
four rational points in general position, and fix group isomorphisms
 $\tilde N(\bar k)\simeq A_5$
and $\Aut(Y)\simeq S_5$
(see \cite[Theorem~8.5.8]{dolgachevclassical}).
As $\Aut(A_5)=S_5$, the natural action of~$\Gal(\bar k/k)$ on~$\tilde N(\bar k)$ determines
a homomorphism $\chi:\Gal(\bar k/k)\to S_5$.  
Letting~$S_5$ act on~$A_5$ by conjugation, the twist by~$\chi$
of the constant group scheme over~$k$ associated with~$A_5$ is~$\tilde N$.
Let~$Y'$ denote the twist of~$Y$ by~$\chi$.
As the action of~$A_5$ on~$Y$ is $S_5$\nobreakdash-equivariant, it gives rise, upon twisting,
to an action of~$\tilde N$ on~$Y'$.
Let us now consider the diagonal right action of the group scheme~$\tilde N$ on $\SL_n \times_k Y$ and the two
projections $\mathrm{pr}_1:(\SL_n \times_k Y')/\tilde N \to \SL_n/\tilde N$
and $\mathrm{pr}_2:(\SL_n \times_k Y')/\tilde N \to Y'/\tilde N$.
As the generic fibre of~$\mathrm{pr}_2$ is a torsor under the rational algebraic group~$\SL_n$,
it is itself rational, by Hilbert's Theorem~90.
As the generic fibre of~$\mathrm{pr}_1$ is a del Pezzo surface of degree~$5$,
it is also rational, by the work of Enriques, Manin and Swinnerton-Dyer
 (see \cite{enriques,maninrational,sd-del-pezzo-5}).
The varieties $\SL_n/\tilde N$ and~$Y'/\tilde N$ are therefore stably birationally equivalent.
To conclude the proof, let us check that the surface $Y'/\tilde N$ is rational.
Let $Z \to Y'/\tilde N$ denote its minimal resolution of singularities.
As~$Y'$ is a del Pezzo surface of degree~$5$, the above-cited
work of Enriques, Manin and Swinnerton-Dyer
implies that~$Y'$ is rational and hence that $Z(k)\neq \emptyset$.
On the other hand, according to Trepalin \cite[end of proof of Lemma~4.5]{trepalin}, the smooth projective surface~$Z_{\bar k}$ is isomorphic to the Hirzebruch surface $\P(\sO_{\P^1}\oplus \sO_{\P^1}(3))$ over~$\bar k$. In particular $K_Z^2=8$
and~$Z$ is geometrically minimal, and hence
minimal.  All in all~$Z$ is a minimal smooth projective geometrically rational surface with~$K_Z^2 \geq 5$
and $Z(k)\neq\emptyset$;
by the work of Iskovskikh and Manin,
it follows that it is rational (see \cite[Proposition~4.16]{bwch2xk}), as desired.
\end{proof}

\subsubsection{Inverse Galois problem}
\label{subsec:inversegalois}

In the situation of Theorem~\ref{th:genthb},
let us assume that~$X$ is the quotient of~$L$ by a finite subgroup~$\Gamma \subseteq L(k)$
viewed as a constant group scheme over~$k$, and let~$\bar x$
 be
the image of $1 \in L(k)$, so that $H_{\bar x} = \Gamma$. Then any normal subgroup $N \subseteq \Gamma$ is
stable under the (trivial) outer action of $\Gal(\bar k/k)$,
and 
the supersolvability condition on~$\Gamma/N$ that appears
in the statement of Theorem~\ref{th:genthb}
reduces to the usual notion of supersolvability for abstract groups.
When in addition $L=\SL_n$, the conclusion
of Theorem~\ref{th:genthb} implies a positive answer
to the inverse Galois problem for~$\Gamma$
(see \cite[\textsection4, Proposition~1]{harariquelques})
and, by a theorem of Lucchini Arteche,
to the Grunwald problem outside of the finite places of~$k$ dividing the order of~$\Gamma$
(see \cite[\textsection6]{lucchiniunramifiedbrauer}).  We shall refer to this version
of the Grunwald problem as the \emph{tame} Grunwald problem,
following \cite[\textsection1.2]{dlan}.
Thus, we obtain:

\begin{cor}
\label{cor:inversegalois}
Let~$\Gamma$ be a finite group and $N \subset \Gamma$ be a normal subgroup such that
$\Gamma/N$ is supersolvable.  Let~$k$ be a number field.
Assume that
the set $Y(k)$ is dense
in $Y(k_{\Omega})^{\Br_{\nr}(Y)}$
 for any $n\geq 1$ and any homogeneous space~$Y$ of~$\SL_n$ over~$k$
whose geometric stabilisers are isomorphic, as groups, to~$N$.
Then~$\Gamma$ is a Galois
group over~$k$ and the tame Grunwald problem has a positive solution
  for~$\Gamma$ over~$k$.
\end{cor}

The second assertion of the corollary means that
if~$S$ is a finite set of places of~$k$ none of which divides
the order of~$\Gamma$ and if, for each $v \in S$, a Galois extension $K_v/k_v$
whose Galois group can be embedded into~$\Gamma$
is given,
then
there exists a Galois extension~$K/k$ with Galois group~$\Gamma$ such that
for each $v \in S$, the completion of~$K$ at a place dividing~$v$ is isomorphic,
as a field extension of~$k_v$, to~$K_v$.

When the subgroup~$N$ is assumed to be trivial, the statement of Corollary~\ref{cor:inversegalois}
 recovers
the positive answer to the tame Grunwald problem
for supersolvable finite groups
obtained in
 \cite[Corollaire au théorème~B]{hwzceh}.
Indeed, in this case, Hilbert's Theorem~90
guarantees that $Y\simeq \SL_n$,
so that~$Y$ is rational over~$k$ and satisfies the weak approximation property.

Combining
 Corollary~\ref{cor:symmetricalternating} for $G=N=A_5$
(due to Boughattas and Neftin~\cite{boughattasneftin})
with
Corollary~\ref{cor:inversegalois}
 yields the following case of the
tame Grunwald  problem, which to our knowledge is new:

\begin{cor}
Let $\Gamma = A_5 \rtimes G$ be a semi-direct product of~$A_5$ with a finite supersolvable group~$G$. 
Then~$\Gamma$ is a Galois
group over~$k$ and the tame Grunwald problem has a positive solution for~$\Gamma$ over~$k$.
\end{cor}

A more precise understanding of the unramified Brauer group
of $\SL_n/\Gamma$  leads, for some groups~$\Gamma$, to a positive answer to the Grunwald
problem with a smaller exceptional set
of bad primes than in the conclusion of Corollary~\ref{cor:inversegalois}
(see \cite[4th page]{hwzceh}).
If one does not exclude
any prime, however, the Grunwald problem
can  have a negative answer 
 under the assumptions of Corollary~\ref{cor:inversegalois}, as is well known
(see \cite{wang}).

\subsection{Norms from supersolvable extensions}\label{subsec:norms}

The following
refinement of the inverse Galois problem
was formulated by
Frei, Loughran and Newton~\cite{freiloughrannewton}:
given a number field~$k$,
a finite group~$G$ and
a finitely generated subgroup $\mathcal A \subset k^*$,
does there exist a Galois extension $K/k$
such that $\Gal(K/k)\simeq G$ and $\mathcal A \subset N_{K/k}(K^*)$?
We provide a positive answer when~$G$ is supersolvable.

\begin{thm}
\label{th:normfromsupersolvable}
Let~$k$ be a number field and $\mathcal A \subset k^*$ be a finitely generated subgroup.
Let~$G$ be a supersolvable finite group.
There exists a Galois extension $K/k$ with Galois group isomorphic to~$G$
such that every element of~$\mathcal A$ is a norm from~$K$.
Moreover, given a finite set of places~$S$ of~$k$, one can require that the places of~$S$ split in~$K$.
\end{thm}

\begin{proof}
Let us fix
an embedding $G \hookrightarrow \SL_n(k)$ for some~$n\geq 1$
and a finite system of generators $\alpha_1,\dots,\alpha_m$ of~$\mathcal A$.
As in \cite[\textsection A.1]{freiloughrannewton},
we
let $T^\alpha \subset \prod_{g \in G} \Gm$,
for $\alpha \in k^*$,
denote the
subvariety, over~$k$, whose $\bar k$\nobreakdash-points are the maps $t:G \to \bar k^*$
such that $\prod_{g \in G}t(g)=\alpha$.
Let us set $Y = \SL_n \times T^{\alpha_1} \times \dots \times T^{\alpha_m}$
and $L = \SL_n \times T^1 \times \dots \times T^1$
(with~$m$ copies of~$T^1$).
Let~$G$ act
on the right on~$T^\alpha$ (for any~$\alpha$) by $(t\cdot g)(g')=t(gg')$.
Let~$G$ act
on the right on~$Y$ and on~$L$ by the diagonal actions
coming from the action just defined on the~$T^\alpha$,
from the right multiplication action on the copy of~$\SL_n$ appearing in~$Y$,
and from the trivial action on the copy of~$\SL_n$ appearing in~$L$.

Set $X=Y/G$.
The projection $\pi:Y\to X$ is a right torsor under~$G$
since the action of~$G$ on~$\SL_n$
by right multiplication is free.
We view it as a left torsor by setting $g \cdot y = y \cdot g^{-1}$.
Applying
Corollary~\ref{cor:quotientsRC}
as in Example~\ref{example:quotientlinear}
shows that $X(k)$ is a dense subset of $X(k_\Omega)^{\Br_{\nr}(X)}$.
It follows that Ekedahl's version of Hilbert's irreducibility theorem,
in the form spelled out in \cite[Lemme~1]{harariquelques}, can be applied to~$\pi$.
Let us fix $y_0 \in Y(k)$.
Recall that there exists
a  finite subset $S_0 \subset \Omega$
 such that any $(x_v)_{v \in \Omega} \in X(k_\Omega)$
with $x_v=\pi(y_0)$ for all $v \in S_0$
belongs to
$X(k_{\Omega})^{\Br_{\nr}(X)}$
(see \cite[Remarks~2.4 (i)--(ii)]{wittenbergslc}).
Let us fix such an~$S_0$, and an~$S$ as in the statement of Theorem~\ref{th:normfromsupersolvable}.
By \cite[Lemme~1]{harariquelques},
there exists $x \in X(k)$ such that the scheme $\pi^{-1}(x)$ is irreducible,
with~$x$
 arbitrarily close
to $\pi(y_0) \in X(k_v)$ for all $v \in S$.

The function field~$K$ of $\pi^{-1}(x)$
is then a Galois extension of~$k$ with group~$G$,
and the restriction to $\pi^{-1}(x)$ of the
invertible function $(s,t_1,\dots,t_m) \mapsto t_i(1)$ on~$Y$,
where~$1$ denotes the identity element of~$G$,
is an element of~$K^*$ with norm~$\alpha_i$.
Moreover, as~$\pi$ is étale, the implicit function theorem guarantees that if~$x$ is sufficiently close to~$\pi(y_0)$ in~$k_v$ for $v \in S_0$, then $\pi^{-1}(x)$ possesses a $k_v$\nobreakdash-point (close to~$y_0$), so that
 $v$ splits in~$K$.
\end{proof}

\begin{rmks}
(i)
In the particular case where the group~$G$ is abelian,
the existence of Galois extensions~$K/k$
such that $\Gal(K/k)\simeq G$ and $\mathcal A \subset N_{K/k}(K^*)$
was first established
by Frei, Loughran and Newton \cite[Theorem~1.1]{freiloughrannewton},
who gave a quantitative estimate for the number of such field extensions with bounded
conductor. In \emph{op.\ cit.}, Appendix, we offered an alternative
algebro-geometric proof of their result.
A third proof
was later given by Frei and Richard~\cite{freirichard}.

(ii)
In the particular case where the group~$G$ is abelian,
the proof of
 \cite[Appendix]{freiloughrannewton}
is simpler
than the one obtained by specialising the above proof of Theorem~\ref{th:normfromsupersolvable}.
Indeed, the latter chooses a filtration of~$G$
with cyclic quotients and proceeds by induction along this filtration,
while the former consists in one step only.

(iii)
When formulating Theorem~\ref{th:normfromsupersolvable}, one might consider
local conditions at a finite set of places~$S$ more general than the condition
that these places split in~$K$;
in
 the abelian case,
this was done in \cite[Corollary~4.12]{freiloughrannewton}, \cite[\textsection2.4]{freirichard}.
Even when a
 Galois extension of~$k$ with group~$G$
that satisfies the given local conditions is assumed to exist,
arbitrary local conditions
cannot always be satisfied when the group~$G$
is a non-abelian finite supersolvable group---see~\cite[Proposition~A.9]{freiloughrannewton} for an example.  
According to the proof of Theorem~\ref{th:normfromsupersolvable},
this phenomenon is fully controlled by the Brauer--Manin obstruction to weak approximation on the
variety~$X$ that appears in this proof.
In particular, there always exists a finite subset $T \subset \Omega$
such that arbitrary local conditions can be imposed at the places of~$S$ as soon as $S\cap T=\emptyset$.
To identify~$T$ explicitly, however, it would be necessary to analyse further the group $\Br_{\nr}(X)$.
\end{rmks}

The ideas underlying the proof of
Theorem~\ref{th:normfromsupersolvable} can be applied
to other similar problems.
To conclude the article, we explain a slightly more general framework and give one example.

We fix a number field~$k$, a finite group~$G$ and a subgroup $H \subseteq G$
such that the only normal subgroup of~$G$ contained in~$H$ is the trivial subgroup,
and we let~$G$
act on the polynomial ring $k\big[(x_{\gamma})_{\gamma \in G/H}\big]$ by permuting the variables via
$g(x_\gamma)=x_{g\gamma}$.
We also fix a non-constant invariant polynomial
 $\theta \in k\big[(x_{\gamma})_{\gamma \in G/H}\big]^G$.

Let us consider a $k$\nobreakdash-algebra~$\tilde K$ endowed with an action of~$G$ that turns the morphism
$\Spec(\tilde K) \to \Spec(k)$ into a $G$\nobreakdash-torsor.  Let $K = \tilde K^H$.
When~$\tilde K$ is a field, this amounts to specifying a field extension~$K/k$, with Galois
closure~$\tilde K/k$, together with a group isomorphism $p:G \isoto \Gal(\tilde K/k)$
such that $p(H)=\Gal(\tilde K/K)$.

For $z \in K$ and $\gamma \in G/H$, if $\tilde\gamma \in G$ stands for a lift of~$\gamma$, the
element $\tilde\gamma(z)\in\tilde K$
does
not depend on the choice of~$\tilde\gamma$. We denote it by $\gamma(z)$.
For $z \in K$,
substituting~$\gamma(z)$ for $x_\gamma$ in~$\theta$ yields an element of~$\tilde K$
that is invariant under~$G$ and hence belongs to~$k$.  We denote it by $N_{\theta}(z)$.
This defines a map $N_{\theta}:K\to k$.  For example, if $\theta=\prod_{\gamma} x_\gamma$
(resp.\ $\theta=\sum_{\gamma} x_\gamma$), we recover the norm (resp.\ trace) map from~$K$ to~$k$.

Given a finite subset $\mathcal A \subset k$, one can now ask: do there exist a field extension $K/k$,
a Galois closure~$\tilde K/k$ and
 an isomorphism~$p$ as above, such that $\mathcal A \subset N_{\theta}(K)$?  When $H$ is trivial
and $\theta=\prod_\gamma x_\gamma$, this is exactly the question considered in
\cite{freiloughrannewton,freirichard}
and in Theorem~\ref{th:normfromsupersolvable}.  The proof of Theorem~\ref{th:normfromsupersolvable} extends
to a positive answer to this more general question under some assumptions on~$G$, $H$ and~$\theta$,
which we present in Theorem~\ref{th:normfromsupersolvablegen} below.
To prepare for its statement, let us introduce, for $\alpha \in k$, the affine variety
$V^\alpha=\Spec\big(
k\big[(x_{\gamma})_{\gamma \in G/H}\big] / (\theta-\alpha)
\big)$.
As~$\theta$ is invariant under~$G$, the group~$G$ naturally acts on~$V^\alpha$.
We set  $V= \prod_{\alpha\in\mathcal A}V^\alpha$
and equip this product with the diagonal action of~$G$.

\begin{thm}
\label{th:normfromsupersolvablegen}
Let~$k$ be a number field.
Let~$G$ be a supersolvable finite group and $H \subseteq G$ be a subgroup such that 
 the only normal subgroup of~$G$ contained in~$H$ is the trivial subgroup.
Let
 $\theta \in k\big[(x_{\gamma})_{\gamma \in G/H}\big]^G$
be a non-constant invariant polynomial.
 Let $\mathcal A \subset k$ be a finite subset.
Let~$V$ be the variety associated with $H$, $G$, $\theta$, $\mathcal A$ as above.
Let $\Spec(\tilde K_0)\to \Spec(k)$ be a $G$\nobreakdash-torsor.
Assume that the following conditions are satisfied:
\begin{enumerate}
\item letting $K_0=\tilde K_0^H$, the inclusion
 $\mathcal A \subset N_\theta(K_0)$ holds;
\item the variety~$V$ is smooth and rationally connected;
\item for every $[\sigma] \in H^1(k,G)$,
 Conjecture~\ref{conj:ct} holds for the twisted variety~$V^\sigma$.
\end{enumerate}
Then there exist a field extension~$K/k$, with Galois closure $\tilde K/k$,
and a group isomorphism $p:G \isoto \Gal(\tilde K/k)$, such that $p(H)=\Gal(\tilde K/K)$
and $\mathcal A \subset N_\theta(K)$.
If moreover a finite set of places~$S$ of~$k$ is given, one can require that
for all $v \in S$,
 the $k_v$\nobreakdash-algebras
 $\tilde K \otimes_k k_v$ and $\tilde K_0 \otimes_k k_v$ are $G$\nobreakdash-equivariantly isomorphic
(so that
 the $k_v$\nobreakdash-algebras
 $K \otimes_k k_v$ and $K_0 \otimes_k k_v$ are isomorphic).
\end{thm}

\begin{rmks}
\label{rmks:vsigmatwist}
(i)
If $[\sigma_0] \in H^1(k,G)$ denotes the class of the torsor $\Spec(\tilde K_0)\to\Spec(k)$,
the twisted variety $V^{\sigma_0}$ admits a rational point if and only if $\mathcal A \subset N_{\theta}(K_0)$.
To explain why, we first note that
the twist of the affine space
$\Spec\big(k\big[(x_{\gamma})_{\gamma \in G/H}\big]\big)$
by~$\sigma_0$ can be identified with the Weil restriction
$R_{K_0/k}\A^1_{K_0}$.  As the regular function $\theta$ on this affine space
is invariant under~$G$, it induces a regular function
on its twist. We view it as a morphism $N_\theta:R_{K_0/k}\A^1_{K_0} \to \A^1_k$.
For $\alpha \in \mathcal A$, the twist of~$V^\alpha$ by~$\sigma_0$
is then the fibre $N_\theta^{-1}(\alpha)$.  The latter possesses a rational point if and only
if $\alpha \in N_\theta(K_0)$; hence the claim.

(ii)
The group~$G$ acts faithfully, and therefore generically freely, on~$V$.
Indeed~$G$ acts faithfully on~$G/H$ by our assumption on~$H$,
hence it also acts faithfully on
$k\big[(x_{\gamma})_{\gamma \in G/H}\big]$, while~$\theta-\alpha$ is not
a scalar multiple of $x_{\gamma_1}-x_{\gamma_2}$ for any $\gamma_1,\gamma_2\in G/H$.

(iii)
Let $V'$ be the largest open subset of~$V$ on which~$G$ acts freely.
If $(V'/G)(k)\neq\emptyset$, then
 a $G$\nobreakdash-torsor $\Spec(\tilde K_0) \to \Spec(k)$ satisfying assumption~(1)
of Theorem~\ref{th:normfromsupersolvablegen} exists.
Indeed, for $c \in (V'/G)(k)$, twisting the $G$\nobreakdash-torsor $V' \to V'/G$
by its fibre 
 $\Spec(\tilde K_0) \to \Spec(k)$
above~$c$
yields
a $G$\nobreakdash-torsor $(V')^{\sigma_0} \to V'/G$
whose total space admits rational points (namely, rational points above~$c$).
Assumption~(1) then holds by Remark~\ref{rmks:vsigmatwist}~(i).
\end{rmks}

\begin{proof}[Proof of Theorem~\ref{th:normfromsupersolvablegen}]
We adapt the proof of Theorem~\ref{th:normfromsupersolvable} as follows.
Fix
an embedding $G \hookrightarrow \SL_n(k)$ for some~$n\geq 1$.
Set $Y=\SL_n \times V$.
Let~$G$ act by right multiplication on~$\SL_n$, by the given action on~$V$,
and  diagonally on~$Y$.  Set $X=Y/G$. Let $\tilde\pi:Y\to X$
and $\pi:Y/H\to X$ denote the quotient maps.

\begin{lem}
There exists $x_0 \in X(k)$ such that $\tilde \pi^{-1}(x_0)$
and $\Spec(\tilde K_0)$
are $G$\nobreakdash-equivariantly
isomorphic  over~$k$.
\end{lem}

\begin{proof}
Let us consider the cartesian square
\begin{align*}
\xymatrix@C=3em{
Y \ar[r]^(.46){\tilde \pi} \ar[d]_{\pr_1} & X \ar[d]^{\pr_1/G} \\
\SL_n \ar[r]^(.46){\tilde\rho} & \SL_n/G\rlap.
}
\end{align*}
As the set $H^1(k,\SL_n)$ is a singleton (Hilbert's Theorem~90),
there exists $b \in (\SL_n/G)(k)$
such that $\tilde\rho^{-1}(b)$ is $G$\nobreakdash-equivariantly isomorphic to $\Spec(\tilde K_0)$
(apply \cite[Chapitre~I, \textsection5.4, Corollaire~1]{serrecg} to the inclusion $G \hookrightarrow \SL_n(\bar k)$).  The fibre $(\pr_1/G)^{-1}(b)$ is then isomorphic to the twist of~$V$ by the torsor
$\Spec(\tilde K_0)\to \Spec(k)$.
By Remark~\ref{rmks:vsigmatwist}~(i), we deduce from this and from our assumption~(1) that
 $(\pr_1/G)^{-1}(b)$ possesses a rational point, say~$x_0$.
As $\tilde\pi^{-1}(x_0)=\tilde\rho^{-1}(b)$, the lemma is proved.
\end{proof}

Assumptions~(2) and~(3) allow us to deduce from
Corollary~\ref{cor:quotientsRC}
that $X(k)$ is a dense subset of $X(k_\Omega)^{\Br_{\nr}(X)}$.
Therefore there exists
 $x \in X(k)$ such that the scheme $\tilde \pi^{-1}(x)$ is irreducible,
with~$x$
 arbitrarily close
to $x_0 \in X(k_v)$ for all $v \in S$
(see \cite[Lemme~1]{harariquelques}).
Let~$\tilde K$ and~$K$  denote the function fields
of $\tilde\pi^{-1}(x)$  and $\pi^{-1}(x)$,
respectively.
The field extension~$\tilde K/k$ is Galois with group~$G$,
and we have $\Gal(\tilde K/K)=H$ by construction.
By choosing~$x$ sufficiently close to~$x_0$ for $v \in S$, we can ensure that
for all $v \in S$,
 the $k_v$\nobreakdash-algebras
 $\tilde K \otimes_k k_v$ and $\tilde K_0 \otimes_k k_v$ are $G$\nobreakdash-equivariantly isomorphic
(see \cite[Lemma~4.6]{harskononab}).
It remains to check that 
 $\mathcal A \subset N_{\theta}(K_0)$.
For $\alpha \in \mathcal A$, composing the projection map $Y \to V^\alpha$
with the regular function on~$V^\alpha$ given by $x_H$ (where~$H$ denotes the canonical point of $G/H$)
yields a regular function on~$Y$ that is invariant under~$H$, hence descends to a regular function on~$Y/H$.
Its restriction $z\in K$ to~$\pi^{-1}(x)$ satisfies $N_{\theta}(z)=\alpha$, as desired.
\end{proof}

\begin{example}
\label{ex:th416}
When
$\theta=\prod_\gamma x_\gamma$
and $\mathcal A \subset k^*$,
the varieties~$V^\alpha$ are trivial torsors under trivial tori,
so that~$V$ is rational over~$k$ and assumptions~(1) and~(2) of
 Theorem~\ref{th:normfromsupersolvablegen} both hold,
in view of 
 Remark~\ref{rmks:vsigmatwist}~(i),
if one takes for $\Spec(\tilde K_0)\to \Spec(k)$ the trivial torsor.
Assumption~(3) holds as well, as the twisted varieties $(V^\alpha)^\sigma$
are torsors under tori
 (see Example~\ref{example:quotientlinear}).
When in addition~$H$ is the trivial subgroup,
this recovers
Theorem~\ref{th:normfromsupersolvable}.
\end{example}

\begin{example}
\label{ex:cubicnoncyclic}
Let~$k$ be a number field and $\alpha \in k^*$.
Then there exists a cubic
 extension $K/k$
such that the equation $\alpha=\mathrm{Tr}_{K/k}(\beta^2)$ has a solution $\beta \in K$.

To see this, one
applies
Theorem~\ref{th:normfromsupersolvablegen} to the symmetric group $G=S_3$,
to the subgroup~$H$ generated by
a transposition,
 to $\theta=\sum_\gamma x_\gamma^2$ and to $\mathcal A=\{\alpha\}$.
Let us check that its hypotheses hold.
To this end, we identify $G/H$ with $\{1,2,3\}$,
so that~$V^\alpha$ is the smooth affine quadric surface defined by the equation
$x_1^2+x_2^2+x_3^2=\alpha$.
The twisted varieties
$(V^\alpha)^\sigma \subset (\A^3_k)^\sigma$ are also smooth affine quadric surfaces, since $(\A^3_k)^\sigma \simeq \A^3_k$ (Hilbert's Theorem~90).
Smooth quadric surfaces are rationally connected and satisfy the
weak approximation property, hence assumptions~(2) and~(3) of Theorem~\ref{th:normfromsupersolvablegen}
are satisfied.
To verify the existence of a torsor $\Spec(\tilde K_0)\to \Spec(k)$ satisfying~(1), we consider
the point $(x_1,x_2,x_3)=(0,\sqrt{\alpha/2},-\sqrt{\alpha/2})$.
In the notation of Remark~\ref{rmks:vsigmatwist}~(iii),
this point belongs to $V'(\bar k) \subset V^{\alpha}(\bar k)$
since its coordinates are pairwise distinct. As its orbit under $\Gal(\bar k/k)$ is contained in its orbit
under~$G$, we have $(V'/G)(k)\neq\emptyset$: Remark~\ref{rmks:vsigmatwist}~(iii) can be applied.
\end{example}

\begin{rmk}
We note that the cubic extensions constructed in Example~\ref{ex:cubicnoncyclic} are non-cyclic.
A cyclic cubic extension~$K/k$ such that the equation $\alpha=\Tr_{K/k}(\beta^2)$ has a solution $\beta\in K$
need not exist: for instance, it cannot exist if~$\alpha$ is not totally positive.
This is a situation where Theorem~\ref{th:normfromsupersolvablegen} does not apply because there
is no torsor $\Spec(\tilde K_0)\to \Spec(k)$ satisfying its assumption~(1)
(taking $G=\Z/3\Z$ and letting~$H$ be trivial).
Similarly, even in the non-cyclic case, it is not always possible to ensure
that the places of a finite set $S\subset \Omega$ split completely in~$K$:
for instance, this cannot be achieved if~$S$ contains a real place
at which~$\alpha$ is negative.
Here $(V'/G)(k)\neq\emptyset$, as shown in Example~\ref{ex:cubicnoncyclic}, but $V(k)=\emptyset$.
These two observations exhibit a marked
contrast with the situation considered in Theorem~\ref{th:normfromsupersolvable}
and in Example~\ref{ex:th416}, where $V(k)\neq\emptyset$ automatically.
\end{rmk}

\bibliographystyle{amsalpha}
\bibliography{supersolvable}
\end{document}